\documentclass[11pt,reqno]{amsart}

\usepackage{amsmath,amssymb,mathrsfs,amsthm,mathabx}
\usepackage{graphicx,cite,cases}
\newtheorem{theorem}{Theorem}[section]

\newtheorem{lemma}[theorem]{Lemma}

\numberwithin{equation}{section}

\begin{document}

\title[Convergence of an adaptive finite element DtN method]{Convergence of an adaptive finite element DtN method for the elastic wave scattering by periodic
structures}

\author{Peijun Li}
\address{Department of Mathematics, Purdue University, West Lafayette, Indiana
47907, USA.}
\email{lipeijun@math.purdue.edu}

\author{Xiaokai Yuan}
\address{Department of Mathematics, Purdue University, West Lafayette, Indiana
47907, USA.}
\email{yuan170@math.purdue.edu}


\subjclass[2010]{78A45, 65N12, 65N15, 65N30}

\keywords{Elastic wave equation, scattering by periodic structures, adaptive finite element method, transparent boundary condition, DtN map, a posteriori estimate.}

\begin{abstract}
Consider the scattering of a time-harmonic elastic plane wave by a periodic rigid surface. The elastic wave propagation is governed
by the two-dimensional Navier equation. Based on a Dirichlet-to-Neumann (DtN)
map, a transparent boundary condition (TBC) is introduced to reduce the
scattering problem into a boundary value problem in a bounded domain. By using
the finite element method,  the discrete problem is considered, where the TBC is
replaced by the truncated DtN map. A new duality argument is developed to derive
the a posteriori error estimate, which contains both the finite element
approximation error and the DtN truncation error. An a posteriori error estimate
based adaptive finite element algorithm is developed to solve the elastic
surface scattering problem. Numerical experiments are presented to demonstrate
the effectiveness of the proposed method. 
\end{abstract}

\maketitle

\section{Introduction}

The scattering theory in periodic structures, which are known as gratings in
optics, has many significant applications in micro-optics including the design
and fabrication of optical elements such as corrective lenses, anti-reflective
interfaces, beam splitters, and sensors \cite{P-80, BCM-SIAM-2001}. Driven by
the optical industry applications, the time-harmonic scattering problems have
been extensively studied for acoustic and electromagnetic waves in periodic
structures. We refer to \cite{BDC-JOSAA-1995, CF-TAMS-1991} and the references
cited therein for the mathematical results on well-posedness of the solutions
for the diffraction grating problems. Computationally, various numerical methods
have been developed, such as boundary integral equation method
\cite{WL-JOSA-2009, NS-SJMA-1991}, finite element method \cite{B-SJNA-1995,
B-SJAM-1997}, boundary perturbation method \cite{BR-JOSA-1993}. Recently, the
scattering problems for elastic waves have received much attention due to the
important applications in seismology and geophysics \cite{A-M2AS-99,
A-JIEA-1999, LWZ-IP-2015}. This paper concerns the scattering of a time-harmonic
elastic plane wave by a periodic surface. Compared with acoustic and
electromagnetic wave equations, the elastic wave equation is less studied due to
the complexity of the coexistence of compressional and shear waves with
different wavenumbers. In addition, there are two challenges for the
scattering problem: the solution may have singularity due to a possible
nonsmooth surface; the problem is imposed in an open domain. In this paper, we
intend to address both issues.

The first issue can be overcome by using the a posteriori error estimate based
adaptive finite element method. A posteriori error estimates are computable
quantities from numerical solutions and measure the solution errors of discrete problems without
requiring any a priori information of real solutions \cite{BR-SJNA-1978,
M-JCAM-1998}. They are crucial in designing numerical algorithms for mesh
modification such as refinement and coarsening \cite{D-SJNA-1996, V-1996}. The aim is to equidistribute the
computational effort and optimize the computation. The a posteriori error
estimate based adaptive finite element method has the ability of error control
and asymptotically optimal approximation property \cite{CKNS-SJNA-2008,
MNS-SJNA-2000}. It has become an important numerical tool for solving
differential equations, especially for those where the solutions have
singularity or multiscale phenomena.

The second issue concerns the domain truncation. The surface scattering problem
is imposed in an open domain, which needs to be truncated into a bounded
computational domain. An appropriate boundary condition is required on the
boundary of the truncated domain so that no artificial wave reflection occurs. 
Such a boundary condition is called a non-reflecting boundary condition or a
transparent boundary condition (TBC) \cite{BT-CPAM-1980, EM-MC-1977, GG-WM-2003, GK-WM-1990, GK-JCP-1995}. Despite a huge amount of work
done so far in this aspect, it still remains to be one of the important and active research topics in the computational wave propagation. 
Since B\'{e}renger proposed a perfectly matched layer (PML) technique to solve Maxwell's equations
\cite{B-JCP-1994}, the research on PML has undergone a tremendous development
due to its effectiveness and simplicity \cite{BP-MC-07, BPT-MC-10,
BW-SJNA-2005, CW-MOTL-1994, CM-SISC-1998, CT-G-2001, HSB-JASA-1996,
HSZ-SIMA-2003, LS-C-1998, TY-ANM-1998}. Various constructions of PML have been
proposed for solving a wide range of wave propagation problems. The idea of PML
technique is to surround the domain of interest by a layer of finite thickness
of fictitious medium that may attenuate the waves coming from inside of the
computational domain. When the waves reach the outer boundary of the PML region,
their amplitudes are so small that the homogeneous Dirichlet boundary
condition can be imposed.

Combined with the PML technique, an adaptive finite element method was proposed
in \cite{CW-SINUM-2003} to solve the two-dimensional diffraction grating
problem. It was shown that the a posteriori error estimate consists of the
finite element discretization error and the PML truncation error which decays
exponentially with respect to the PML parameters. Due to the competitive
numerical performance, the methods was quickly extended to solve the two- and
three-dimensional obstacle scattering problems \cite{CL-SINUM-2005, CC-MC-2008}
and the three-dimensional diffraction grating problem \cite{BPW-MC-2010}. Based
on the a posteriori error analysis, the adaptive finite element PML method
provides an effective numerical strategy to solve a variety of acoustic,
electromagnetic, and elastic wave propagation problems which are imposed in
unbounded domains \cite{CXZ-MC-2016, JLLZ-MMNA-2017}.

The Dirichlet-to-Neumann (DtN) method is another approach to handle the domain truncation. The idea is to construct an explicit solution, which is usually given as an infinite Fourier series, in the exterior of the domain of interest. By taking the normal derivative of the solution, the Neumann data can be expressed in terms of the Dirichlet data. This relationship gives the DtN map and can be used as a boundary condition, which is known as the TBC. Since the TBC is exact, the artificial boundary can be put as close as possible to the scattering structures, which can reduce the size of the computational domain.

Recently, as a viable alternative to the PML, the adaptive finite element DtN method has been proposed to solve the scattering problems imposed in open domains, such as the obstacle scattering problems \cite{JLLZ-JSC-2017, JLZ-CCP-2013}, the diffraction grating problems \cite{WBLLW-SJNA-2015}. In this approach, the TBC is applied on the artificial boundary which is chosen to enclose the domain of interest. These TBCs
are based on nonlocal DtN maps and are given as infinite Fourier series. Practically, the infinite series needs to be truncated into the sum of finite number of terms by choosing an appropriate truncation parameter $N$. It is
known that the convergence of the truncated DtN map could be arbitrarily slow to
the original DtN map in the operator norm \cite{HNPX-JCAM-2011}. To overcome
this issue, the duality argument has to be developed to obtain the a posteriori
error estimate between the solution of the scattering problem and the finite
element solution. Comparably, the a posteriori error estimates consists of the
finite element discretization error and the DtN truncation error, which decays exponentially
with respect to the truncation parameter $N$.

In this paper, we present an adaptive finite element DtN method for the elastic wave scattering problem in periodic
structures. The goal is threefold: (1) prove the exponential convergence of the
truncated DtN  operator; (2) give a complete a  posteriori  error estimate; (3)
develop an effective adaptive finite element algorithm.  This paper
significantly extends the work on the acoustic scattering problem \cite{WBLLW-SJNA-2015}, where the
Helmholtz equation was considered. Apparently, the techniques differ greatly
from the existing work because of the complicated transparent boundary condition
associated with the elastic wave equation. A related work can be found in \cite{LY-2019} for an
adaptive finite element DtN method for solving the obstacle scattering problem
of elastic waves.

Specifically, we consider the scattering of an elastic plane wave by a one-dimensional rigid
periodic surface, where the wave motion is governed by the two-dimensional Navier equation. 
The open space above the surface is assumed
to be filled with a homogeneous and isotropic elastic medium. The Helmholtz
decomposition is utilized to reduce the elastic wave equation equivalently into
a coupled boundary value problem of the Helmholtz equation. By combining the quasi-periodic boundary condition and a DtN
operator, an exact TBC is introduced to reduce the original scattering problem into a boundary
value problem of the elastic wave equation in a bounded domain. The discrete
problem is studied by using the finite element method with the truncated DtN
operator. Based on the Helmholtz decomposition, a new duality argument is
developed to obtain an a posteriori error estimate between the solution of the
original scattering problem and the discrete problem. The a posteriori error
estimate contains the finite element approximation error and the DtN operator truncation
error, which is shown to decay exponentially with respect to the
truncation parameter. The estimate is used to design the adaptive finite element
algorithm to choose elements for refinements and to determine the truncation
parameter $N$. Due to the exponential convergence of the truncated DtN operator, the choice of the truncation parameter $N$ is not sensitive to the given tolerance. Numerical experiments are presented to demonstrate the effectiveness
of the proposed method.

The outline of the paper is as follows. In Section \ref{section:pf}, the model equation is introduced for the scattering problem. In Section \ref{section:bvp}, the boundary value problem is formulated by using the TBC and the corresponding weak formulation is studied. In Section \ref{section:dp}, the discrete problem is considered by using the finite element method with the truncated DtN operator. Section \ref{section:pea} is devoted to the a posterior error estimate. In Section \ref{section:ne}, we discuss the numerical implementation of the adaptive algorithm and present two examples to illustrate the performance of the proposed method. The paper is concluded with some general remarks and directions for future work in Section \ref{section:c}.

\section{Problem formulation}\label{section:pf}

Consider the scattering of a time-harmonic plane wave by an elastically rigid surface, which is assumed to be invariant in the $z$-axis and periodic in the $x$-axis with period $\Lambda$. Due to the periodic structure, the problem can be restricted into a single periodic cell where $x\in (0, \Lambda)$. Let $\boldsymbol x=(x, y)\in\mathbb R^2$. Denote the surface by $S=\{\boldsymbol x\in\mathbb R^2: y=f(x),\, x\in(0, \Lambda)\},$ where $f$ is a Lipschitz continuous function. Let $\nu$ and $\tau$ be the unit normal and tangent vectors on $S$, respectively. Above $S$, the open space is assumed to be filled with a homogeneous and isotropic elastic medium with unit mass density. Denote $\Omega_f^+=\{\boldsymbol x\in\mathbb R^2 : y>f(x),\, x\in (0, \Lambda)\}.$ Let $\Gamma=\{\boldsymbol x\in \mathbb{R}^2: y=b,\, x\in(0, \Lambda)\}$ and $\Gamma'=\{\boldsymbol x\in \mathbb{R}^2: y=b',\, x\in (0, \Lambda)\}$, where $b$ and $b'$ are constants satisfying $b>b'>\max_{x\in (0, \Lambda)} f(x)$.  Denote 
$\Omega=\{\boldsymbol x\in\mathbb R^2: f(x)<y<b,\,x\in (0, \Lambda)\}$. The problem geometry is shown in Figure \ref{pg}.

\begin{figure}
\centering
\includegraphics[width=0.3\textwidth]{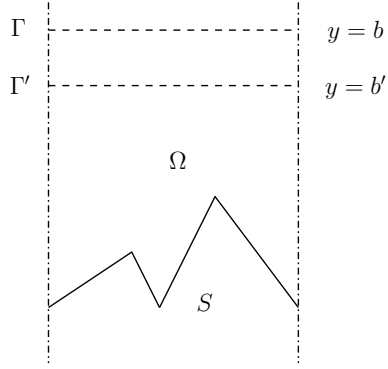}
\caption{Schematic of the elastic wave scattering by a
periodic structure.}
\label{pg}
\end{figure}

The incident wave $\boldsymbol u^{\rm inc}$ satisfies the two-dimensional
elastic wave equation
\[
 \mu\Delta\boldsymbol u^{\rm inc}+(\lambda+\mu)\nabla\nabla\cdot\boldsymbol
u^{\rm inc}+\omega^2\boldsymbol u^{\rm inc}=0\quad\text{in} ~ \Omega_f^+,
\]
where $\omega>0$ is the angular frequency and $\mu, \lambda$ are the Lam\'{e}
parameters satisfying $\mu>0, \lambda+\mu>0$. Specifically, the incident wave
can be the compressional plane wave $\boldsymbol u^{\rm
inc}(\boldsymbol x)=\boldsymbol d e^{{\rm i}\kappa_1\boldsymbol
x\cdot\boldsymbol d}$ or the shear plane wave $\boldsymbol u^{\rm
inc}(\boldsymbol x)=\boldsymbol d^\perp e^{{\rm i}\kappa_2\boldsymbol
x\cdot\boldsymbol d}$, where $\boldsymbol d=(\sin\theta, -\cos\theta)^\top,
\boldsymbol d^\perp=(\cos\theta, \sin\theta)^\top, \theta=(-\pi/2, \pi/2)$ is the
incident angle, $\kappa_1=\omega/(\lambda+2\mu)^{1/2}$ and
$\kappa_2=\omega/\mu^{1/2}$ are known as the compressional and shear
wavenumbers, respectively. For clarity, we shall take the compressional plane
wave as the incident field. The results will be similar if the incident field is
the shear plane wave.

Due to the interaction between the incident wave and the surface, the
scattered wave is generated and satisfies
\begin{equation}\label{uNavier}
\mu\Delta\boldsymbol{u}+(\lambda+\mu)\nabla\nabla\cdot\boldsymbol{u}
+\omega^2\boldsymbol{u}=0 \quad \text{in} ~ \Omega_f^+.
\end{equation}
Since the surface $S$ is elastically rigid, the displacement of the
total field vanishes and the scattered field satisfies
\begin{equation}\label{uBoundary}
\boldsymbol{u}=-\boldsymbol{u}^{\rm inc}\quad \text{on} ~ S.
\end{equation}

For any solution $\boldsymbol{u}$ of \eqref{uNavier}, it has the Helmholtz
decomposition 
\begin{equation}\label{uHelmholtz}
\boldsymbol{u}=\nabla\phi_1+{\bf curl} \phi_2,
\end{equation}
where $\phi_j, j=1, 2$ are scalar potential functions and ${\bf
curl}\phi_2=(\partial_y\phi_2, -\partial_x\phi_2)^\top$. Substituting
\eqref{uHelmholtz} into \eqref{uNavier}, we may verify that $\phi_j$ satisfies the Helmholtz equation
\begin{equation}\label{HelmholtzDec}
\Delta \phi_j+\kappa_j^2 \phi_j=0 \quad\text{in} ~ \Omega^+_f.
\end{equation}
Taking the dot product of \eqref{uBoundary} with $\nu$ and
$\tau$, respectively, yields that 
\begin{equation*}
\partial_{\nu}\phi_1-\partial_{\tau}\phi_2=\boldsymbol{u}^{\rm inc}\cdot
\nu, \quad \partial_{\nu}\phi_2+\partial_{\tau}\phi_1=-\boldsymbol{u}^{\rm
inc}\cdot\tau\quad \text{on}~ S.
\end{equation*}

Let $\alpha=\kappa_{\rm p}\sin\theta$. It is clear to note that
$\boldsymbol{u}^{\rm inc}$ is a quasi-periodic function with respect to $x$,
i.e., $\boldsymbol{u}^{\rm inc}(x,y) e^{-{\rm i}\alpha x}$
is a periodic function with respect to $x$. Motivated by uniqueness of the solution, we
require that the solution $\boldsymbol u$ of
\eqref{uNavier}--\eqref{uBoundary} is also a quasi-periodic function of $x$
with period $\Lambda$. 

We introduce some notations and functional spaces. Let $H^1(\Omega)$ be the
standard Sobolev space. Denote a quasi-periodic functional space
\[
H_{\rm qp}^{1}(\Omega)=\{ u\in H^{1}(\Omega): u(\Lambda, y)=u(0, y)e^{{\rm
i}\alpha \Lambda}\}.
\]
Let $H^{1}_{S, {\rm qp}}(\Omega)=\{u\in H^1_{\rm qp}(\Omega): u=0\text{ on } S\}$. 
Clearly, $H_{\rm qp}^{1}(\Omega)$ and $H^{1}_{S, {\rm qp}}(\Omega)$ are
subspaces of $H^1(\Omega)$ with the standard $H^{1}$-norm. For
any function $u\in H_{\rm qp}^{1}(\Omega)$, it admits the Fourier expansion on $\Gamma$:
\[
u(x, b)=\sum_{n\in \mathbb{Z}}u^{(n)}(b)e^{{\rm i}\alpha_n x},
\quad u^{(n)}(b)=\frac{1}{\Lambda}\int_{0}^{\Lambda} u(x, b) e^{-{\rm
i}\alpha_n x}{\rm d}x,\quad
\alpha_n=\alpha+n\left(\frac{2\pi}{\Lambda}\right).
\]	
The trace functional space $H^{s}(\Gamma), s\in \mathbb{R}$ is defined by
\[
H^{s}(\Gamma)=\left\{u\in L^2(\Gamma): \|u\|_{H^{s}(\Gamma)}<\infty
\right\},
\] 
where the norm is given by
\[
\|u\|_{H^{s}(\Gamma)}=\left(\Lambda \sum\limits_{n\in
\mathbb{Z}}\left(1+\alpha_n^2\right)^s 
|u^{(n)}(b)|^2\right)^{1/2}.
\] 
Let $\boldsymbol{H}_{{\rm qp}}^{1}(\Omega), \boldsymbol{H}_{S, {\rm
qp}}^{1}(\Omega), \boldsymbol{H}^{s}(\Gamma)$ be the
Cartesian product spaces equipped with the corresponding
2-norms of $H^1_{\rm qp}(\Omega), H_{S, {\rm qp}}^{1}(\Omega), H^{s}(\Gamma)$,
respectively. Throughout the paper, the notation $a\lesssim b$ stands for $a\leq
C b$, where $C$ is a positive constant whose value is not
required but should be clear from the context.

\section{The boundary value problem}\label{section:bvp}

The scattering problem \eqref{uNavier}--\eqref{uBoundary} is formulated in the
open domain $\Omega^{+}_f$, which needs to be truncated into the bounded
domain $\Omega$. An appropriate boundary condition is required on
$\Gamma$ to avoid artificial wave reflection.

Let $\phi_j$ be the solution of the Helmholtz equation
\eqref{HelmholtzDec} along with the bounded outgoing wave condition.
It is shown in \cite{LWZ-IP-2015} that $\phi_j$ is a 
quasi-periodic function and admits the Fourier series expansion
\begin{equation}\label{TBC:Helmholtz}
\phi_j(x,y)=\sum\limits_{n\in \mathbb{Z}} \phi_j^{(n)}(b) e^{{\rm
i}\left(\alpha_n x+ \beta_j^{(n)} (y-b)\right)},\quad y>b,
\end{equation}
where
\begin{equation}\label{beta}
\beta_j^{(n)}=
\begin{cases} 
\left(\kappa_j^2-\alpha_n^2\right)^{1/2}, & |\alpha_n|<\kappa_j,\\
{\rm i}\left(\alpha_n^2-\kappa_j^2\right)^{1/2}, &|\alpha_n|>\kappa_j.
\end{cases}
\end{equation}
We assume that $\kappa_j\neq |\alpha_n|$ for $n\in \mathbb{Z}$ to exclude
possible resonance. Taking the normal derivative of \eqref{TBC:Helmholtz}
on $\Gamma$ yields
\begin{eqnarray*}
\partial_y \phi_j(x, b)=\sum\limits_{n\in\mathbb{Z}} {\rm
i}\beta_j^{(n)}\phi_j^{(n)}(b)e^{{\rm i}\alpha_n x}. 
\end{eqnarray*}

As a quasi-periodic function, the solution $\boldsymbol u(x, y)=(u_1(x, y),
u_2(x, y))^\top$ admits the Fourier expansion
\[
\boldsymbol{u}(x,y)=\sum\limits_{n\in\mathbb{Z}}(u_1^{(n)}(y),
u_2^{(n)}(y))^\top e^{{\rm i}\alpha_n x}, \quad y>b,
\]
where $u_j^{(n)}$ is the Fourier coefficient of $u_j$. Define a boundary
operator
\[
\mathscr B\boldsymbol{u}=\mu\partial_y \boldsymbol{u}+(\lambda+\mu)(0,1)^\top 
\nabla\cdot\boldsymbol{u}\quad\text{on} ~ \Gamma. 
\]
It is shown in \cite{JLLZ-MMNA-2017} that the solution of
\eqref{uNavier} satisfies the transparent boundary condition
\begin{equation}\label{TBC: uTBC}
\mathscr B\boldsymbol{u}=\mathscr{T}\boldsymbol{u}:=\sum\limits_{n\in\mathbb{Z}}
M^{(n)}(u_1^{(n)}(b), u_2^{(n)}(b))^\top e^{{\rm i}\alpha_n
x}\quad\text{on} ~ \Gamma, 
\end{equation}
where $\mathscr T$ is called the Dirichlet-to-Neumann (DtN) operator and $M^{(n)}$ is a $2\times 2$ matrix given by 
\begin{equation}\label{TBCMatrix}
M^{(n)}=\frac{\rm i}{\chi_n}\begin{bmatrix}
\omega^2\beta_1^{(n)} & \mu\alpha_n\chi_n-\omega^2\alpha_n\\[2pt]
\omega^2\alpha_n-\mu\alpha_n\chi_n &
\omega^2\beta_2^{(n)}\end{bmatrix}.
\end{equation}
Here $\chi_{n}=\alpha_n^2+\beta_1^{(n)}\beta_2^{(n)}$.

By the transparent boundary condition \eqref{TBC: uTBC}, the variational problem
of \eqref{uNavier}--\eqref{uBoundary} is to find $\boldsymbol{u}\in
\boldsymbol{H}_{{\rm qp}}^{1}(\Omega)$
with $\boldsymbol{u}=-\boldsymbol{u}^{\rm inc}$ on $S$ such that
\begin{equation}\label{TBC: variational1}
a(\boldsymbol{u}, \boldsymbol{v})=0, \quad \forall \boldsymbol{v}\in
\boldsymbol{H}_{S, {\rm qp}}^{1}(\Omega),
\end{equation}
where the sesquilinear form $a: \boldsymbol H^1_{{\rm
qp}}(\Omega)\times\boldsymbol H^1_{{\rm qp}}(\Omega)\to\mathbb C $ is
defined as 
\begin{align*}
a(\boldsymbol{u}, \boldsymbol{v})=\mu\int_{\Omega}\nabla\boldsymbol{u}
:\nabla\overline{\boldsymbol{v}}{\rm d}\boldsymbol{x}
+(\lambda+\mu)\int_{\Omega}\left(\nabla\cdot\boldsymbol{u}
\right)\left(\nabla\cdot\overline{\boldsymbol{v}}\right){\rm d}\boldsymbol{x}\\
-\omega^2\int_{\Omega}\boldsymbol{u}\cdot\overline{\boldsymbol{v}}
{\rm d}\boldsymbol{x}-\int_{\Gamma}\mathscr{T}\boldsymbol{u}\cdot\overline{
\boldsymbol {v}}{\rm d}s.
\end{align*}
Here $A:B={\rm tr}(AB^\top)$ is the Frobenius inner product of two square
matrices $A$ and $B$.

The well-posedness of the variational problem \eqref{TBC: variational1} was
discussed in \cite{EH-MMAS-2010}. It was shown that the variational problem 
\eqref{TBC: variational1} has a unique solution for all frequencies if the
surface $S$ is Lipschitz continuous.  Hence we may assume that the variational
problem \eqref{TBC: variational1} admits a unique solution and the solution
satisfies the estimate
\begin{equation}\label{uBabu}
\|\boldsymbol{u}\|_{\boldsymbol{H}^{1}(\Omega)}\lesssim
\|\boldsymbol{u}^{\rm inc}\|_{\boldsymbol{H}^{1/2}(S)}\lesssim
\|\boldsymbol{u}^{\rm inc}\|_{\boldsymbol{H}^{1}(\Omega)}.
\end{equation}
By the general theory of Babuska and Aziz \cite{BA-AP-1973}, there exists
$\gamma>0$ such that the following inf-sup condition holds
\[
\sup\limits_{0\neq \boldsymbol{v}\in  \boldsymbol{H}_{{\rm qp}}^{1}(\Omega)}
\frac{|a(\boldsymbol{u}, \boldsymbol{v})|}{\|\boldsymbol{v}\|_{
\boldsymbol{H}^{1}(\Omega)}} \geq \gamma 
\|\boldsymbol{u}\|_{\boldsymbol{H}^{1}(\Omega)},
\quad \forall \boldsymbol{u}\in  \boldsymbol{H}_{{\rm qp}}^{1}(\Omega).
\]

\section{The discrete problem}\label{section:dp}

We consider the discrete problem of \eqref{TBC: variational1} by using the
finite element approximation. Let $\mathcal M_h$ be a regular triangulation of
$\Omega$, where $h$ denotes the maximum diameter of all the elements in
$\mathcal M_h$. Since our focus is on the a posteriori error estimate, for
simplicity, we assume that $S$ is polygonal and ignore the approximation error
of the boundary $S$. Thus any edge $e\in\mathcal M_h$ is a
subset of $\partial \Omega$ if it has two boundary vertices. Moreover, we
require that if $(0, y)$ is a node on the left boundary, then $(\Lambda, y)$ is
also a node on the right boundary and vice versa, which allows to define a
finite element space whose functions are quasi-periodic respect to $x$.  

Let $\boldsymbol V_h\subset \boldsymbol{H}_{\rm qp}^{1}(\Omega)$ be a
conforming finite element space, i.e.,
\[
\boldsymbol V_h:= \left\{\boldsymbol v\in C(\overline{\Omega})^2:
\boldsymbol v|_{K}\in P_m(K)^2 \text{ for any } K\in\mathcal M_h,\,
\boldsymbol v(0,y)=e^{-{\rm i}\alpha \Lambda}\boldsymbol v(\Lambda, y)\right\},
\]
where $m$ is a positive integer and $P_m(K)$ denotes the set of all polynomials
of degree no more than $m$. The finite element approximation to the variational
problem \eqref{TBC: variational1} is to find $\boldsymbol u^h\in\boldsymbol
V_h$ with $\boldsymbol u^h=-\boldsymbol u^{\rm inc}$ on $S$ such that 
\begin{equation}\label{TBC: variational2}
 a(\boldsymbol u^h, \boldsymbol v^h)=0,\quad\forall\boldsymbol v^h\in
\boldsymbol V_{h, S},
\end{equation}
where $\boldsymbol V_{h, S}=\{\boldsymbol v\in\boldsymbol V_h: \boldsymbol
v=0\text{ on } S\}$. 

In the variational problem \eqref{TBC: variational2}, the boundary operator
$\mathscr{T}$ is defined as an infinite series, in practice, it must be
truncated to a sum of finitely many terms as follows 
\begin{equation}\label{TBC: tTBC}
\mathscr{T}_N \boldsymbol{u}=\sum\limits_{|n|\leq
N}M^{(n)}(u_1^{(n)}(b), u_2^{(n)}(b))^\top e^{{\rm i}\alpha_n x}, 
\end{equation}
where $N>0$ is a sufficiently large constant. Using the truncated boundary
operator, we arrive at the truncated finite element approximation: Find
$\boldsymbol{u}_N^h\in\boldsymbol V_{h}$ such that it satisfies
$\boldsymbol{u}_N^h=-\boldsymbol{u}^{\rm inc}$ on
$S$ and the variational problem 
\begin{equation}\label{TBC: variational3}
a_N(\boldsymbol{u}^h_N, \boldsymbol{v}^h)=0,\quad \forall
\boldsymbol{v}^h\in \boldsymbol{V}_{h, S},
\end{equation}
where the sesquilinear form $a_N: \boldsymbol V_h\times\boldsymbol
V_h\to\mathbb C$ is defined as
\begin{align*}
a_N(\boldsymbol{u}, \boldsymbol{v})
=\mu\int_{\Omega}\nabla\boldsymbol{u}:\nabla\overline{\boldsymbol {v}}
{\rm d}\boldsymbol{x}+(\lambda+\mu)\int_{\Omega}(\nabla\cdot\boldsymbol{u})
(\nabla\cdot\overline{\boldsymbol{v}}){\rm d}\boldsymbol{x}\\
-\omega^2 \int_{\Omega}  \boldsymbol{u}\cdot \overline{\boldsymbol{v}}{\rm
d}\boldsymbol{x} -\int_{\Gamma}\mathscr{T}_N\boldsymbol{u}\cdot\overline{
\boldsymbol{v}} {\rm d}s. 
\end{align*}

It follows from \cite{S-MC-74} that the discrete inf-sup condition of the
sesquilinear form $a_N$  can be established for sufficient large $N$ and small
enough $h$. Based on the general theory in \cite{BA-AP-1973}, it can be
shown that the discrete variational problem \eqref{TBC: variational3} has a
unique solution $\boldsymbol u_N^h\in \boldsymbol{V}_{h}$. The details are
omitted for brevity.

\section{The a posteriori error analysis}\label{section:pea}

For any triangular element $K\in\mathcal M_h$, denoted by $h_K$ its diameter.
Let $\mathcal B_h$ denote the set of all the edges of $K$. For any $e\in\mathcal
B_h$, denoted by $h_e$ its length. For any interior edge $e$ which is the common
side of $K_1$ and $K_2\in\mathcal M_h$, we define the jump residual across $e$
as
\[
J_e=\mu\nabla \boldsymbol{u}_N^h|_{K_1}\cdot
\boldsymbol{\nu}_1+(\lambda+\mu)(\nabla\cdot\boldsymbol{u}_N^h|_{K_1})
\boldsymbol{\nu}_1+\mu\nabla \boldsymbol{u}_N^h|_{K_2}\cdot\boldsymbol{\nu}
_2+(\lambda+\mu)(\nabla\cdot\boldsymbol{u}_N^h|_{K_2} )\boldsymbol{\nu}_2,
\]
where $\boldsymbol{\nu}_j$ is the unit outward normal vector on the boundary of
$K_j, j=1,2$. For any boundary edge $e\subset \Gamma$, we define the jump
residual 
\[
J_e=2(\mathscr{T}_N\boldsymbol{u}_N^{h}-\mathscr B\boldsymbol{u}_N^{h}).
\]
For any boundary edge on the left line segment of $\partial\Omega$, i.e.,
$e\in\{x=0\}\cap \partial K_1$ for some $K_1\in\mathcal M_{h}$, and
its corresponding edge on the right line segment of $\partial\Omega$, i.e.,
$e'\in\{x=\Lambda\}\cap \partial K_2$ for some $K_2\in\mathcal M_{h}$, the
jump residual is
\begin{align*}
J_{e} &=\left[\mu\partial_x
\boldsymbol{u}_N^{h}|_{K_1}+(\lambda+\mu)(1,0)^\top
\nabla\cdot\boldsymbol{u}_N^{h}|_{K_1}\right]
-e^{-{\rm i}\alpha\Lambda}\left[\mu\partial_x
\boldsymbol{u}_N^{h}|_{K_2}
+(\lambda+\mu)(1,0)^\top \nabla\cdot\boldsymbol{u}_N^{h}|_{K_2}\right],\\
J_{e'} &=e^{{\rm i}\alpha\Lambda}
\left[\mu\partial_x
\boldsymbol{u}_N^{h}|_{K_1}+(\lambda+\mu)(1,0)^\top
\nabla\cdot\boldsymbol{u}_N^{h}|_{K_1}\right]	
-\left[\mu\partial_x \boldsymbol{u}_N^{h}|_{K_2}
+(\lambda+\mu)(1,0)^\top
\nabla\cdot\boldsymbol{u}_N^{h}|_{K_2}\right].
\end{align*}
For any triangular element $K\in\mathcal M_h$, denote by $\eta_K$ the local
error estimator which is given by 
\[
\eta_K=h_K \|\mathscr{R}\boldsymbol{u}_N^{h}\|_{\boldsymbol{L}^{2}(K)}
+\left(\frac{1}{2}\sum\limits_{e\in\partial K}h_e
\|J_e\|^2_{\boldsymbol{L}^2(e)}\right)^{1/2},
\]
where $\mathscr R$ is the residual operator defined by
\[
\mathscr{R}\boldsymbol{u}=\mu\Delta\boldsymbol{u}
+(\lambda+\mu)\nabla\left(\nabla\cdot\boldsymbol{u}\right)
+\omega^2\boldsymbol{u}.
\]

For convenience, we introduce a weighted norm of
$\boldsymbol{H}^1(\Omega)$ as
\[
\vvvert\boldsymbol{u}\vvvert^2_{\boldsymbol{H}^{1}(\Omega)}
=\mu\int_{\Omega}  |\nabla\boldsymbol{u}|^2{\rm d}\boldsymbol{x}
+(\lambda+\mu)\int_{\Omega}|\nabla\cdot\boldsymbol{u}|^2 {\rm d}\boldsymbol{x}
+\omega^2 \int_{\Omega} |\boldsymbol{u}|^2{\rm d}\boldsymbol{x}.
\]
It is easy to check that 
\begin{equation}\label{Equal_norm}
\min \left(\mu, \omega^2\right)
\|\boldsymbol{u}\|^2_{\boldsymbol{H}^{1}(\Omega)}
\leq \vvvert\boldsymbol{u}\vvvert^2_{\boldsymbol{H}^{1}(\Omega)}
\leq \max\left(2\lambda+3\mu,
\omega^2\right)\|\boldsymbol{u}\|^2_{\boldsymbol{H}^{1}(\Omega)},
\quad \forall \boldsymbol{u}\in \boldsymbol{H}^{1}(\Omega).
\end{equation}
which implies that the weighted norm is equivalent to standard
$\boldsymbol{H}^{1}(\Omega)$ norm.

Now we state the main result of this paper.

\begin{theorem}\label{Main_Result}
Let $\boldsymbol{u}$ and $\boldsymbol{u}_N^h$ be the solutions of the 
variational problem \eqref{TBC: variational1} 
and \eqref{TBC: variational3}, respectively. Then for sufficient large
$N$, the following a posteriori error estimate holds 
\begin{eqnarray*}
\|\boldsymbol{u}-\boldsymbol{u}_N^h
\|_{\boldsymbol{H}^{1}(\Omega)}\lesssim 
\left(\sum\limits_{K\in\mathcal M_h} \eta^2_{K}\right)^{1/2}
+\max_{|n|>N}\left(|n|e^{-|\beta_2^{(n)}|(b-b')}\right)
\|\boldsymbol{u}^{\rm inc}\|_{\boldsymbol{H}^{1}(\Omega)}.
\end{eqnarray*}
\end{theorem}

It is easy to note that the a posteriori error consists of two parts: the
finite element discretization error and the truncation error of the DtN
operator. We point out that the latter is almost exponentially decaying since
$b>b'$ and $|\beta_2^{(n)}|>0$. In practice, the DtN truncated error can be
controlled to be small enough such that it does not contaminate the finite
element discretization error.

In the rest of the paper, we shall prove the a posteriori error estimate in
Theorem \ref{Main_Result}. First, let's state the trace regularity for functions
in $H^{1}_{\rm qp}(\Omega)$. The proof can be found in
\cite{CW-SINUM-2003}.

\begin{lemma}\label{Lemma2}
For any $u\in H^{1}_{\rm qp}(\Omega)$, the following estimates hold
\[
\|u\|_{H^{1/2}(\Gamma_b)}\lesssim \|u\|_{H^{1}(\Omega)},\quad
\|u\|_{H^{1/2}(\Gamma_{b'})}\lesssim \|u\|_{H^{1}(\Omega)}.
\] 
\end{lemma}

Denote by $\boldsymbol{\xi}=\boldsymbol{u}-\boldsymbol{u}_N^h$ the error
between the solutions of \eqref{TBC: variational1} and \eqref{TBC:
variational3}. It can be verified that 
\begin{eqnarray}\label{Xi}
\vvvert\boldsymbol{\xi}\vvvert^2_{\boldsymbol{H}^{1}(\Omega) }
&=& \mu\int_{\Omega} 
\nabla\boldsymbol{\xi}:\nabla\overline{\boldsymbol{\xi}}{\rm d}\boldsymbol{x}
+(\lambda+\mu)\int_{\Omega}\left(\nabla\cdot\boldsymbol{\xi}\right)
\left(\nabla\cdot\overline{\boldsymbol{\xi}}\right){\rm d}\boldsymbol{x}
+\omega^2 \int_{\Omega} \boldsymbol{\xi}\cdot\overline{\boldsymbol{\xi}}{\rm
d}\boldsymbol{x}\notag\\
&=& \Re a(\boldsymbol{\xi}, \boldsymbol{\xi})+2\omega^2
\int_{\Omega}\boldsymbol{\xi}\cdot
\overline{\boldsymbol{\xi}}{\rm d}\boldsymbol{x}
+\Re\int_{\Gamma}\mathscr{T}\boldsymbol{\xi}\cdot
\overline{\boldsymbol{\xi}}{\rm d}s \notag\\
&=& \Re a(\boldsymbol{\xi}, \boldsymbol{\xi})+
\Re\int_{\Gamma}\left(\mathscr{T}
-\mathscr{T}_N\right)\boldsymbol{\xi}\cdot\overline{\boldsymbol{\xi}}{\rm d}s
+2\omega^2 \int_{\Omega}\boldsymbol{\xi}\cdot\overline{\boldsymbol{\xi}}
{\rm d}\boldsymbol{x} +\Re\int_{\Gamma}\mathscr{T}_N
\boldsymbol{\xi}\cdot\overline{\boldsymbol{\xi}}{\rm d}s.	
\end{eqnarray}

In the following, we shall discuss the four terms in the right hand side of
\eqref{Xi}. Lemma \ref{Posterior: Lemma3} gives the error estimate of the
truncated DtN operator. Lemma \ref{Posterior: FirstTwo} presents the a
posteriori error estimate for the finite element approximation and the
truncated DtN operator.

\begin{lemma}\label{Posterior: Lemma3}
Let $\boldsymbol{u}\in \boldsymbol{H}^{1}_{\rm qp}(\Omega)$ be the solution of
the variational problem \eqref{TBC: variational1}. For any $\boldsymbol{v}\in
\boldsymbol{H}^{1}_{\rm qp}(\Omega)$, the following estimate holds:
\[
\left|\int_{\Gamma}
\left(\mathscr{T}-\mathscr{T}_N\right)\boldsymbol{u}\cdot\overline{\boldsymbol{v
}}\,{\rm d}s\right|
\leq C\max\limits_{|n|>N}\left(|n| e^{{\rm
i}\beta_2^{n}(b-b')}\right)
\|\boldsymbol{u}^{\rm
inc}\|_{\boldsymbol{H}^{1}(\Omega)}
\|\boldsymbol{v}\|_{\boldsymbol{H}^{1}(\Omega)},
\]
where $C>0$ is a constant independent of $N$. 
\end{lemma}

\begin{proof}
Using \eqref{uHelmholtz} and \eqref{TBC:Helmholtz} yields 
\[
\phi_j^{(n)}(b)=\phi^{(n)}_j(b')e^{{\rm i}\beta_j^{(n)}(b-b')}.
\]
It follows from the straightforward calculations that we obtain 
\begin{eqnarray}
\begin{bmatrix}
u_1^{(n)}(b) \\[5pt]
u_2^{(n)}(b)
\end{bmatrix}&=&\frac{1}{\chi_n}
\begin{bmatrix}
{\rm i}\alpha_n & {\rm i}\beta_2^{(n)}\\[5pt] 
{\rm i}\beta_1^{(n)} & -{\rm i}\alpha_n
\end{bmatrix}
\begin{bmatrix}
e^{{\rm i}\beta_1^{(n)}(b-b')} & 0\\[5pt] 
0 & e^{{\rm i}\beta_2^{(n)}(b-b')}
\end{bmatrix}
\begin{bmatrix}
-{\rm i}\alpha_n & -{\rm i}\beta_2^{(n)}\\[5pt] 
-{\rm i}\beta_1^{(n)} & {\rm i}\alpha_n
\end{bmatrix}
\begin{bmatrix}
u_1^{(n)}(b')\\[5pt] 
u_2^{(n)}(b')
\end{bmatrix} \notag\\
&:=& P^{(n)}
\begin{bmatrix}
u_1^{(n)}(b') \\[5pt] 
u_2^{(n)}(b')
\end{bmatrix},\label{PN}
\end{eqnarray}
where
\begin{eqnarray*}
P^{(n)}_{11} &=& \frac{1}{\chi_n}\left(\alpha_n^2 e^{{\rm
i}\beta_1^{(n)}(b-b')}+\beta_1^{(n)}\beta_2^{(n)}
e^{{\rm i}\beta_2^{(n)}(b-b')}\right),\\
P^{(n)}_{12} &=& \frac{\alpha_n \beta_2^{(n)}}{\chi_n}\left(e^{{\rm
i}\beta_1^{(n)}(b-b')}-e^{{\rm i}\beta_2^{(n)}(b-b')}\right),\\
P^{(n)}_{21} &=& \frac{\alpha_n \beta_1^{(n)}}{\chi_n}\left(e^{{\rm
i}\beta_1^{(n)}(b-b')}-e^{{\rm i}\beta_2^{(n)}(b-b')}\right),\\
P^{(n)}_{22} &=& \frac{1}{\chi_n}\left(\alpha_n^2 e^{{\rm
i}\beta_2^{(n)}(b-b')}+\beta_1^{(n)}\beta_2^{(n)}
e^{{\rm i}\beta_1^{(n)}(b-b')}\right).
\end{eqnarray*}

It is clear to note from \eqref{beta} that $\beta_j^{(n)}$ is purely imaginary
for sufficiently large $|n|$. By the mean value theorem, for sufficiently large
$|n|$, there exists  $\tau\in ({\rm i}\beta_1^{(n)}, {\rm i}\beta_2^{(n)})$ such
that
\begin{eqnarray*}
\chi_n P^{(n)}_{11} &=&
\left(\alpha_n^2+\beta_1^{(n)}\beta_2^{(n)}\right) e^{{\rm
i}\beta_1^{(n)}(b-b')}
+\beta_1^{(n)}\beta_2^{(n)}\left(e^{{\rm
i}\beta_2^{(n)}(b-b')}-e^{{\rm i}\beta_1^{(n)}(b-b')}\right),\\
&=& \left(\alpha_n^2+\beta_1^{(n)}\beta_2^{(n)}\right)
e^{{\rm i}\beta_1^{(n)}(b-b')}
+\beta_1^{(n)}\beta_2^{(n)}(b-b'){\rm
i}(\beta_2^{(n)}-\beta_1^{(n)})e^{\tau (b-b')}.
\end{eqnarray*}
A simple calculation yields 
\begin{eqnarray*}
\alpha_n^2+\beta_1^{(n)}\beta_2^{(n)}&=&
\alpha_n^2-(\alpha_n^2-\kappa_1^2)^{1/2}(\alpha_n^2-\kappa_2^2)^{1/2}\\
&=&\frac{\alpha_n^2\left(\kappa_1^2+\kappa_2^2\right)-\kappa_1^2\kappa_2^2}{
\alpha_n^2+(\alpha_n^2-\kappa_1^2)^{1/2}
(\alpha_n^2-\kappa_2^2)^{1/2}}<\kappa_1^2+\kappa_2^2
\end{eqnarray*}
and
\begin{eqnarray*}
{\rm i}\beta_2^{(n)}-{\rm
i}\beta_1^{(n)}&=&(\alpha_n^2-\kappa_1^2)^{1/2}-(\alpha_n^2-\kappa_2^2)^{1/2}\\
&=&\frac{\kappa_2^2-\kappa_1^2}{(\alpha_n^2-\kappa_1^2)^{1/2}+(
\alpha_n^2-\kappa_2^2)^{1/2}}
<\frac{\kappa_2^2-\kappa_1^2}{2(\alpha_n^2-\kappa_2^2)^{1/2}}.
\end{eqnarray*}
which give
\begin{eqnarray}\label{AsymMn}
|P^{(n)}_{11}|\lesssim e^{{\rm
i}\beta_1^{(n)}(b-b')}+|n|e^{\tau (b-b')}
\lesssim |n|e^{{\rm i}\beta_2^{(n)}(b-b')}.
\end{eqnarray}
Similarly, we may show that 
\[
|P^{(n)}_{ij}|\lesssim |n| e^{{\rm i}\beta_2^{(n)}(b-b')},\quad
i,j=1, 2.
\]
Combining the above estimates lead to
\[
|u_{1}^{(n)}(b)|^2+|u_{2}^{(n)}(b)|^2\lesssim 
n^2 e^{2{\rm i}\beta_2^{(n)}(b-b')}\left(|u_{1}^{(n)}(b')|^2+|u_{2}^{(n)}
(b')|^2\right).
\]

By \eqref{TBC: uTBC} and \eqref{TBC: tTBC}, we have from Lemma \ref{Lemma2}
that 
\begin{eqnarray*}
&&\left|\int_{\Gamma}
\left(\mathscr{T}-\mathscr{T}_N\right)\boldsymbol{u}\cdot\overline{\boldsymbol{v
}}{\rm d}s\right| =
\left|\Lambda\sum\limits_{|n|>N} (M^{(n)}\boldsymbol
u^{(n)}(b))\cdot\overline{\boldsymbol v^{(n)}(b)}
\right|\\
\quad&&\lesssim   \sum\limits_{|n|>N} \left|\left(|n|^{\frac{1}{2}}
\boldsymbol{u}^{(n)}(b)\right) \cdot
\left(|n|^{\frac{1}{2}}
\overline{\boldsymbol{v}^{(n)}(b)}\right)\right|\\
\quad&&\lesssim \left(\sum\limits_{|n|>N}
|n|\left(|u_{1}^{(n)}(b)|^2+|u_{2}^{(n)}(b)|^2\right)\right)^{1/2}
\left(\sum\limits_{|n|>N}|n|\left(|v_{1}^{(n)}(b)|^2+|v_{2}^{(n)}
(b)|^2\right)\right)^{1/2} \\
\quad&&\lesssim  \left(\sum\limits_{|n|>N} |n|^{3}e^{2{\rm
i}\beta_2^{(n)}(b-b')}\left(|u_{1}^{(n)}(b')|^2+|u_{2}^{(n)}
(b')|^2\right)\right)^{1/2} 
\|\boldsymbol{v}\|_{\boldsymbol{H}^{1/2}(\Gamma)} \\
\quad&&\lesssim \max\limits_{|n|>N}\left(|n| e^{{\rm
i}\beta_2^{(n)}(b-b')}\right)
\|\boldsymbol{u}\|_{\boldsymbol{H}^{1/2}(\Gamma_{b'})}
\|\boldsymbol{v}\|_{\boldsymbol{H}^{1/2}(\Gamma)}\\	 
\quad&&\lesssim \max\limits_{|n|>N}\left(|n| e^{{\rm
i}\beta_2^{(n)}(b-b')}\right)
\|\boldsymbol{u}\|_{\boldsymbol{H}^{1}(\Omega)}
\|\boldsymbol{v}\|_{\boldsymbol{H}^{1}(\Omega)}.
\end{eqnarray*}
Using \eqref{uBabu}, we get 
\[
\left|\int_{\Gamma}
\left(\mathscr{T}-\mathscr{T}_N\right)\boldsymbol{u}\cdot\overline{\boldsymbol{v
}}{\rm d}s\right|\lesssim \max\limits_{|n|>N}\left(|n| e^{{\rm
i}\beta_2^{(n)}(b-b')}\right)
\|\boldsymbol{u}^{\rm
inc}\|_{\boldsymbol{H}^{1}(\Omega)}
\|\boldsymbol{v}\|_{\boldsymbol{H}^{1}(\Omega)},
\]
which completes the proof. 
\end{proof}

In the following lemmas, the first two terms in \eqref{Xi} are estimated.

\begin{lemma}\label{Posterior: FirstTwo}
Let $\boldsymbol{v}$ be any function in
$\boldsymbol{H}^{1}_{S, {\rm qp}}(\Omega)$, the following estimate holds
\begin{equation*}
\left|a(\boldsymbol{\xi},
\boldsymbol{v})+\int_{\Gamma}\left(\mathscr{T}-\mathscr{T}_N\right)
\boldsymbol{\xi}\cdot \overline{\boldsymbol{v}}{\rm d}s 
\right|\lesssim
\left(\left(\sum\limits_{K\in \mathcal M_n}\eta_{K}^2\right)^{1/2} 
+\max\limits_{|n|>N}\left(|n| e^{{\rm
i}\beta_2^{(n)}(b-b')}\right)
\|\boldsymbol{u}^{\rm
inc}\|_{\boldsymbol{H}^{1}(\Omega)} \right)
\|\boldsymbol{v}\|_{\boldsymbol{H}^{1}(\Omega)}.
\end{equation*}
\end{lemma} 

\begin{proof}
For any function $\boldsymbol v\in H^1_{S, {\rm qp}}(\Omega)$, we have
\begin{eqnarray*}
&& a(\boldsymbol{\xi}, \boldsymbol{v})
+\int_{\Gamma}\left(\mathscr{T}-\mathscr{T}_N\right)\boldsymbol{\xi}\cdot
\overline{\boldsymbol{v}}{\rm d}s  =a(\boldsymbol{u},
\boldsymbol{v})-a(\boldsymbol{u}_N^{h}, \boldsymbol{v})
+\int_{\Gamma}\left(\mathscr{T}-\mathscr{T}_N\right)\boldsymbol{\xi}\cdot
\overline{\boldsymbol{v}}{\rm d}s \\
&&\qquad =a(\boldsymbol{u},
\boldsymbol{v})-a_N^{h}(\boldsymbol{u}_N^h, \boldsymbol{v})
+a_N^{h}(\boldsymbol{u}_N^h, \boldsymbol{v})
-a(\boldsymbol{u}_N^{h}, \boldsymbol{v})
+\int_{\Gamma}\left(\mathscr{T}-\mathscr{T}_N\right)\boldsymbol{\xi}\cdot
\overline{\boldsymbol{v}}{\rm d}s \\ 
&&\qquad =a(\boldsymbol{u},
\boldsymbol{v})-a_N^{h}(\boldsymbol{u}_N^h, \boldsymbol{v}^h)
-a_N^{h}(\boldsymbol{u}_N^h,
\boldsymbol{v}-\boldsymbol{v}^h)
+\int_{\Gamma}\left(\mathscr{T}-\mathscr{T}_N\right)\boldsymbol{u}_N^h\cdot
\overline{\boldsymbol{v}}{\rm d}s \\
&&\qquad\qquad
+\int_{\Gamma}\left(\mathscr{T}-\mathscr{T}_N\right)\boldsymbol{\xi}\cdot
\overline{\boldsymbol{v}}{\rm d}s \\
&&\qquad =-a_N^h(\boldsymbol{u}_N^h,
\boldsymbol{v}-\boldsymbol{v}^h)
+\int_{\Gamma}\left(\mathscr{T}-\mathscr{T}_N\right)\boldsymbol{u}\cdot
\overline{\boldsymbol{v}}{\rm d}s.
\end{eqnarray*}
For any function $\boldsymbol{v}\in \boldsymbol{H}_{S, {\rm qp}}^{1}(\Omega)$
and $\boldsymbol{v}^h\in\boldsymbol{V}_{h, S}$, it follows from the integration
by parts that 
\begin{eqnarray}\label{ftw-s1}
&& -a_N^h(\boldsymbol{u}_N^h, \boldsymbol{v}-\boldsymbol{v}^h) \notag\\
&=& -\sum\limits_{K\in\mathcal M_h}\left\{ \mu\int_{K}
\nabla\boldsymbol{u}_N^h:\nabla \big(\overline{\boldsymbol
v}-\overline{\boldsymbol{v}^h}\big){\rm d}\boldsymbol{x}
+(\lambda+\mu)\int_{K}
(\nabla\cdot\boldsymbol{u}_N^h)
\nabla\cdot\big(\overline{\boldsymbol v}-\overline{\boldsymbol{v}^h}\big)
{\rm d}\boldsymbol{x}\right\} \notag\\
&&-\sum\limits_{K\in\mathcal
M_h}\left\{-\omega^2\int_{K}\boldsymbol{u}_N^h\cdot 
\big(\overline{\boldsymbol v}-\overline{\boldsymbol{v}^h}\big){\rm
d}\boldsymbol{x} -\int_{\Gamma\cap \partial K}
\mathscr{T}\boldsymbol{u}_N^{h}\cdot\big(
\overline{\boldsymbol v}-\overline{\boldsymbol{v}^h}\big){\rm d}s\right\}
\notag\\
&=& \sum\limits_{K\in\mathcal M_h} \left\{ -\int_{\partial
K}\left[\mu\nabla\boldsymbol{u}_N^h\cdot\boldsymbol{\nu}
+(\lambda+\mu)(\nabla\cdot\boldsymbol{u}_N^h)\boldsymbol{\nu}\right]
\cdot\big(\overline{\boldsymbol v}-\overline{\boldsymbol{v}^h}\big)
{\rm d}\boldsymbol{x} +\int_{\Gamma\cap \partial
K}\mathscr{T}\boldsymbol{u}_N^h\cdot
\big(\overline{\boldsymbol v}-\overline{\boldsymbol{v}^h}\big){\rm
d}s\right\}\notag\\
&& +\sum\limits_{K\in\mathcal M_h}\int_{K}
\left[\mu\Delta\boldsymbol{u}_N^h
+(\lambda+\mu)\nabla\nabla\cdot\boldsymbol{u}_N^h+\omega^2\boldsymbol{u}
_N^h\right]\cdot \big(\overline{\boldsymbol
v}-\overline{\boldsymbol{v}^h}\big){\rm d}\boldsymbol{x} \notag\\
&=&  \sum\limits_{K\in\mathcal M_h}\left[\int_{K} \mathscr{R}\boldsymbol{u}_N^h
\cdot \big( \overline{\boldsymbol v}-\overline{\boldsymbol{v}^h}\big){\rm
d}\boldsymbol{x} +\sum\limits_{e\in \partial K}\frac{1}{2}\int_{e}J_e
\cdot \big( \overline{\boldsymbol v}-\overline{\boldsymbol{v}^h}\big){\rm d}s 
\right].
\end{eqnarray}
We take $\boldsymbol{v}^h=\Pi_h \boldsymbol{v}\in\boldsymbol V_{h, S}$, where
$\Pi_h$ is the Scott--Zhang interpolation operator and has the following
interpolation estimates 
\begin{eqnarray*}
\|\boldsymbol{v}-\Pi_h
\boldsymbol{v}\|_{\boldsymbol{L}^2(K)}\lesssim h_K
\|\nabla\boldsymbol{v}\|_{\boldsymbol{L}^2(\tilde{K})},\quad
\|\boldsymbol{v}-\Pi_h\boldsymbol{v}\|_{\boldsymbol{L}^2(e)}\lesssim
h_e^{1/2}\|\boldsymbol{v}\|_{\boldsymbol{H}^{1}(\tilde{K}_e)}.
\end{eqnarray*}
Here $\tilde{K}$ and $\tilde{K}_e$ are the unions of all the triangular elements
in $\mathcal M_h$, which have nonempty intersection with the element $K$
and the side $e$, respectively. By the H\"{o}lder equality, we get from
\eqref{ftw-s1} that
\begin{eqnarray*}
|a_N^h(\boldsymbol{u}_N^h, \boldsymbol{v}-\boldsymbol{v}^h)|
\lesssim \left(\sum\limits_{K\in
\mathcal M_h}\eta_K^2\right)^{1/2}\|\boldsymbol{v}\|_{\boldsymbol{H}^1(\Omega)},
\end{eqnarray*}	
which completes the proof.	
\end{proof}

\begin{lemma}\label{Lemma4}
Let $\hat{M}^{(n)} = -\frac{1}{2}(M^{(n)}+ (M^{(n)})^*)$, where $M^{(n)}$ is
defined in \eqref{TBCMatrix}. Then $\hat{M}^{(n)}$ is positive definite for
sufficiently large $|n|$. 
\end{lemma}	

\begin{proof}
It follows from \eqref{beta} that $\beta_j^{(n)}$ is purely imaginary for
sufficiently large $|n|$. By \eqref{TBCMatrix}, we have 
\begin{eqnarray*}
\hat{M}^{(n)} = -\frac{1}{\chi_n}\begin{bmatrix}
{\rm i}\omega^2\beta_1^{(n)} & {\rm
i}\left(\mu\alpha_n\chi_n-\omega^2\alpha_n\right)\\
{\rm i}\left(\omega^2\alpha_n-\mu\alpha_n\chi_n\right) & {\rm
i}\omega^2\beta_2^{(n)}
\end{bmatrix}.
\end{eqnarray*}
Since
$\chi_n=\alpha_n^2-(\alpha_n^2-\kappa_1^2)^{1/2}(\alpha_n^2-\kappa_2^2)^{1/2}
>0$, we get
\[
\hat{M}^{(n)}_{11}=-\frac{\rm i}{\chi_n}\omega^2\beta_1^{(n)}
=\frac{\omega^2}{\chi_n} (\alpha_n^2-\kappa_1^2)^{1/2}>0.
\]
 A simple calculation yields that
\begin{eqnarray*}
\chi_n^2 \det \hat{M}^{(n)} &=&-\omega^4 \beta_1^{(n)}\beta_2^{(n)}
-\left(\mu\alpha_n\chi_n-\omega^2\alpha_n\right)^2\\
&=&-\mu^2\kappa_2^4\left(\chi_n-\alpha_n^2\right)
-\mu^2\alpha_n^2\left(\chi_n-\kappa_2^2\right)^2\\
&=& \mu^2\chi_n
\left(-\kappa_2^4-\alpha_n^2\chi_n+2\alpha_n^2\kappa_2^2\right).
\end{eqnarray*}
Since $\kappa_2>\kappa_1$ and $\alpha^2_n$ has an order of $n^2$ for
sufficiently large $|n|$, we obtain  
\begin{eqnarray*}
2\kappa_2^2-\chi_n
&=&2\kappa_2^2-\alpha_n^2+(\alpha_n^2-\kappa_2^2)^{1/2}
(\alpha_n^2-\kappa_1^2)^{1/2}\\
&=&\kappa_2^2+(\alpha_n^2-\kappa_2^2)^{1/2}\big((\alpha_n^2-\kappa_1^2)^{1/2}
-(\alpha_n^2-\kappa_2^2)^{1/2}\big)>0,
\end{eqnarray*}
which gives that $\det\hat{M}^{(n)}>0$ and completes the proof.
\end{proof}

\begin{lemma}\label{Posterior: Lemma5}
Let $\Omega'=\{\boldsymbol x\in\mathbb R^2: b'<y<b,\,0<x<\Lambda\}$. Then
for any $\delta>0$, there exists a positive constant $C(\delta)$ independent of
$N$ such that 
\begin{equation*}
\Re\int_{\Gamma}\mathscr{T}_N\boldsymbol{\xi}\cdot
\overline{ \boldsymbol{\xi}}{\rm d}s \leq C(\delta)
\|\boldsymbol{\xi}\|^2_{\boldsymbol{L}^2(\Omega')}
+\delta\|\boldsymbol{\xi}\|^2_{\boldsymbol{H}^{1}(\Omega')}.
\end{equation*}
\end{lemma}

\begin{proof}
Using \eqref{TBC: tTBC}, we get from a simple calculation that
\[
\Re\int_{\Gamma}\mathscr{T}_N\boldsymbol{\xi}\cdot \overline{
\boldsymbol{\xi}}{\rm d}s =\Lambda\sum\limits_{|n|\leq
N}\Re\left(M^{(n)}\boldsymbol{\xi}^{(n)}\right)\cdot
\overline{\boldsymbol{\xi}^{(n)}}=-\Lambda\sum\limits_{|n|\leq
N}\left(\hat{M}^{(n)}\boldsymbol{\xi}^{(n)}\right)
\cdot\overline{\boldsymbol{\xi}^{(n)}}. 
\]
By Lemma \ref{Lemma4}, $\hat{M}^{(n)}$ is positive definite for sufficiently
large $|n|$. Hence, for fixed $\omega, \lambda, \mu$, there exists $N^{*}$ such
that $-\left(\hat{M}^{(n)}\boldsymbol{\xi}^{(n)}\right)\cdot\overline{
\boldsymbol{\xi}^{(n)}}\leq 0$ for $n>N^{*}$. Correspondingly, we split
$\Re\int_{\Gamma}\mathscr{T}_N\boldsymbol{\xi}\cdot \overline{
\boldsymbol{\xi}}{\rm d}s$ into two parts:
\begin{equation}\label{Posterior: Lemma5-s1}
\Re\int_{\Gamma}\mathscr{T}_N\boldsymbol{\xi}\cdot
\overline{ \boldsymbol{\xi}}{\rm d}s
=-\Lambda\sum\limits_{|n|\leq \min(N^{*},
N)}\left(\hat{M}^{(n)}\boldsymbol{\xi}_n\right)\cdot\overline{\boldsymbol{\xi}_n
}-\Lambda\sum\limits_{N>|n|>\min(N^{*},
N)}\left(\hat{M}^{(n)}\boldsymbol{\xi}_n\right)\cdot\overline{\boldsymbol{\xi}_n
},
\end{equation}
where $\sum\limits_{N>|n|>\min(N^{*},
N)}\left(\hat{M}^{(n)}\boldsymbol{\xi}_n\right)
\cdot\overline{\boldsymbol{\xi}_n}=0$ if $N>N^{*}$. Since the second part in the
right hand side of \eqref{Posterior: Lemma5-s1} is non-positive, we only need to
estimate the first part in the right hand side of \eqref{Posterior: Lemma5-s1},
which has finitely many terms. Hence there exists 
a constant $C$ depending only on $\omega,\mu, \lambda$ such
that $|\left(\hat{M}^{(n)} \boldsymbol{\xi}^{(n)}\right)\cdot 
\overline{\boldsymbol{\xi}^{(n)}}|\leq C|\boldsymbol{\xi}^{(n)}|^2$ for all
$|n|\leq \min(N^{*}, N)$. 

For any $\delta>0$, it follows from Yong's inequality that 
\begin{eqnarray*}
&& \left(b-b'\right)|\phi(b)|^2 = \int_{b'}^{b}
|\phi(y)|^2{\rm d}y+\int_{b'}^{b}\int_{y}^{b}\left(|\phi(s)|^2\right)'{\rm ds}{\rm d}y\\
&&\qquad \leq\int_{b'}^{b}|\phi(y)|^2{\rm d}y+(b-b')\int_{b'}^{b}2|\phi(y)||\phi'(y)|{\rm d}y
\\
&&\qquad =\int_{b'}^{b}|\phi(y)|^2{\rm d}y+(b-b')\int_{b'}^{b}2\frac{|\phi(y)|}{\sqrt{\delta}}\sqrt{\delta}
|\phi'(y)|{\rm d}y\\
&&\qquad \leq
\int_{b'}^{b}|\phi(y)|^2{\rm d}y+\frac{b-b'}{\delta}\int_{b'}^{b}|\phi(y)|^2{\rm d}y
+\delta(b-b')\int_{b'}^{b}|\phi'(y)|^2{\rm d}y,
\end{eqnarray*}
which gives
\[
|\phi(b)|^2\leq
\left[\frac{1}{\delta}+(b-b')^{-1}\right]\int_{b'}^{b}|\phi(y)|^2{\rm d}y
+\delta\int_{b'}^{b}|\phi^\prime(y)|^2{\rm d}y.
\]
Let $\phi(x,y)=\sum\limits_{n\in\mathbb{Z}}\phi_n(y)e^{{\rm i}\alpha_n
x}$. A simple calculation yields that 
\begin{eqnarray*}
\|\nabla \phi\|^2_{\boldsymbol{L}^2(\Omega')} &=&
\Lambda\sum\limits_{n\in\mathbb{Z}}\int_{b'}^{b}
\left(|\phi'_n(y)|^2+\alpha_n^2|\phi_n(y)|^2\right){\rm d}y,\\
\|\phi\|_{L^{2}(\Omega')}^2  &=&
\Lambda\sum\limits_{n\in\mathbb{Z}}\int_{b'}^{b}|\phi_n(y)|^2{\rm d}y.	
\end{eqnarray*}
Using the above estimates, we have for any $\phi\in H^{1}(\Omega')$ that 
\begin{eqnarray*}
&& \|\phi\|^2_{L^2(\Gamma)} =
\Lambda\sum\limits_{n\in\mathbb{Z}}|\phi_n(b)|^2 \\
&&\,\leq
\Lambda\left[\frac{1}{\delta}+(b-b')^{-1}\right]\sum\limits_{n\in\mathbb{Z}}
\int_{b'}^{b}|\phi_n(y)|^2{\rm d}y
+\Lambda\delta\sum\limits_{n\in\mathbb{Z}}\int_{b'}^{b}|\phi'(y)|^2{\rm d}y\\
&&\,\leq  \Lambda
\left[\frac{1}{\delta}+(b-b')^{-1}\right]
\sum\limits_{n\in\mathbb{Z}}\int_{b'}^{b}
|\phi_n(y)|^2{\rm d}y
+\Lambda
\delta\sum\limits_{n\in\mathbb{Z}}\int_{b'}^{b}\left(|\phi'_n(y)|^2
+\alpha_n^2|\phi_n(y)|^2\right){\rm d}y\\
&&\,\leq 
\left[\frac{1}{\delta}+(b-b')^{-1}\right]\|\phi\|^2_{L^2(\Omega')}
+\delta\|\nabla\phi\|^2_{L^2(\Omega)}\\
&&\, \leq C(\delta)
\|\phi\|^2_{L^2(\Omega')}+\delta\|\nabla \phi\|^2_{L^2(\Omega')}.	
\end{eqnarray*}	
Combining the above estimates, we obtain 
\begin{eqnarray*}
{\rm Re\,}\int_{\Gamma}\mathscr{T}_N\boldsymbol{\xi}\cdot
\overline{ \boldsymbol{\xi}}{\rm d}s 
&\leq&
C\|\boldsymbol{\xi}\|^2_{\boldsymbol{L}^2(\Gamma)}
\leq C(\delta) 
\|\boldsymbol{\xi}\|^2_{\boldsymbol{L}^2(\Omega')}
+\delta\int_{\Omega'}|\nabla\boldsymbol{\xi}|^2{\rm d}\boldsymbol{x}\\
&\leq& C(\delta)
\|\boldsymbol{\xi}\|^2_{\boldsymbol{L}^2(\Omega')}
+\delta\|\boldsymbol{\xi}\|^2_{\boldsymbol{H}^{1}(\Omega')},
\end{eqnarray*}
which completes the proof. 
\end{proof}

To estimate $\int_{\Omega}|\boldsymbol{\xi}|^2{\rm d}\boldsymbol{x}$
in \eqref{Xi} , 
we introduce the dual problem 
\begin{equation}\label{DualProblem}
a(\boldsymbol{v},
\boldsymbol{p})=\int_{\Omega}\boldsymbol{v}\cdot\overline{\boldsymbol{\xi}}
{\rm d}\boldsymbol{x}, \quad \forall \boldsymbol{v}\in \boldsymbol{H}^{1}_{S, {\rm qp}}(\Omega).
\end{equation}
It can be verified that $\boldsymbol{p}$ is the weak solution of the boundary value problem
\begin{equation}\label{dp}
\begin{cases}
\mu\Delta
\boldsymbol{p}+(\lambda+\mu)\nabla\nabla\cdot\boldsymbol{p}+\omega^2\boldsymbol{
p}=-\boldsymbol{\xi} & {\rm in}\,\Omega,\\
 \boldsymbol{p}=0 & {\rm on}\,S,\\
 \mathscr{B}\boldsymbol{p}=\mathscr{T}^{*}\boldsymbol{p} & {\rm on}\,\Gamma,
\end{cases}
\end{equation}
where $\mathscr{T}^{*}$ is the adjoint operator to the DtN operator
$\mathscr{T}$.

It requires to explicitly solve the boundary value problem \eqref{dp}. We consider the Helmholtz decomposition and let
\begin{equation}\label{hdxi}
\boldsymbol{\xi}=\nabla\zeta_1+{\bf curl}\zeta_2,
\end{equation}
where $\zeta_j, j=1, 2$ has the Fourier series expansion
\[
\zeta_j(x,y)=\sum\limits_{n\in \mathbb{Z}}\zeta^{(n)}_j(y) e^{{\rm i}\alpha_n x},\quad b'<y<b.
\]
Consider the following coupled first order ordinary different equations  
\begin{equation*}
\left\{
\begin{aligned}
& \xi^{(n)}_{1}(y)={\rm i}\,\alpha_n
\zeta_1^{(n)}(y)+\zeta_2^{(n)}{}^\prime(y), \\
& \xi^{(n)}_{2}(y)=\zeta_1^{(n)}{}^\prime(y)-{\rm
i}\alpha_n\zeta_2^{(n)}(y),\\
& \zeta_1^{(n)}(b)=0,\quad \zeta_2^{(n)}(b)=0.
\end{aligned}
\right.
\end{equation*}
It follows from straightforward calculations that the solution is
\begin{eqnarray*}
\zeta_1^{(n)}(y) &=& -\frac{\rm
i}{2}e^{\alpha_n(y-b)}\int_{y}^{b}e^{-\alpha_n(t-b)}\xi_n^{(1)}(t){\rm d}t
+\frac{\rm i}{2}e^{-\alpha_n(y-b)}\int_{y}^{b}e^{\alpha_n
(t-b)}\xi_n^{(1)}(t){\rm d}t\\
&&-\frac{1}{2}e^{\alpha_n(y-b)}\int_{y}^{b}e^{-\alpha_n(t-b)}\xi_n^{(2)}
(t){\rm d}t
-\frac{1}{2}e^{-\alpha_n(y-b)}\int_{y}^{b}e^{\alpha_n
(t-b)}\xi_n^{(2)}(t){\rm d}t,\\
\zeta_2^{(n)}(y) &=&
-\frac{1}{2}e^{\alpha_n(y-b)}\int_{y}^{b}e^{-\alpha_n(t-b)}\xi_n^{(1)}(t){\rm d}t
-\frac{1}{2}e^{-\alpha_n(y-b)}\int_{y}^{b}e^{\alpha_n
(t-b)}\xi_n^{(1)}(t){\rm d}t\\
&&+\frac{\rm i}{2}e^{\alpha_n(y-b)}\int_{y}^{b}e^{-\alpha_n(t-b)}\xi_n^{(2)}(t){\rm d}t
-\frac{\rm i}{2}e^{-\alpha_n(y-b)}\int_{y}^{b}e^{\alpha_n
(t-b)}\xi_n^{(2)}(t){\rm d}t.
\end{eqnarray*}	
It is easy to verify the following estimate
\begin{eqnarray*}
\left|\zeta_j^{(n)}(y)\right|\lesssim \left(\|\xi_1^{(n)}\|_{L^{\infty}(b',
b)}+\|\xi_2^{(n)}\|_{L^{\infty}(b', b)}\right)
\frac{1}{|\alpha_n|}e^{|\alpha_n|(b-y)},\quad j=1, 2. 	
\end{eqnarray*}

Let $\boldsymbol p$ be the solution of the dual problem \eqref{dp}. Then it satisfies the following boundary value problem
\begin{equation}\label{tpb}
\begin{cases}
\mu\Delta\boldsymbol{p}+(\lambda+\mu)\nabla\nabla\cdot\boldsymbol{p}
+\omega^2\boldsymbol{p}
=-\boldsymbol{\xi} & \quad {\rm
in}\,\Omega',\\
 \boldsymbol{p}(x, b')=\boldsymbol{p}(x, b') & \quad {\rm
on}\,\Gamma'\\
\mathscr{B}\boldsymbol{p}=\mathscr{T}^{*}\boldsymbol{p} & \quad {\rm on}\,\Gamma.
\end{cases}
\end{equation}
Let function $q_j, j=1, 2$ have the Fourier expansion in $\Omega'$:
\[
q_j(x,y)=\sum\limits_{n\in\mathbb{Z}} q_j^{(n)}(y)e^{{\rm
i}\alpha_n x}.
\]
The Fourier coefficients $q_j^{(n)}$ are required to satisfy the two point
boundary value problem
\begin{equation}\label{Dual_pq}
 \begin{cases}
  q_j^{(n)''}(y)+(\kappa_j^2-\alpha_n^2)q_j^{(n)}(y)=-c_j \zeta_j^{(n)}(y),\\
  q_j^{(n)}(b')=q_j^{(n)}(b'),\\
   q_j^{(n)'}(b)=-{\rm i}\overline{\beta_j^{(n)}}q_j^{(n)}(b),
 \end{cases}
\end{equation}
where $c_1=(\lambda+2\mu)^{-1}$ and $c_2=\mu^{-1}$, $\zeta_j^{(n)}$
are the Fourier coefficients of the potential functions $\zeta_j$
for the Helmholtz decomposition of $\boldsymbol\xi$
in \eqref{hdxi}.

\begin{lemma}\label{Lemma7}
Let $\boldsymbol{p}=\nabla q_1+{\bf curl}q_2$. Then $\boldsymbol p$ satisfies
\eqref{tpb}.
\end{lemma}

\begin{proof}
If \eqref{Dual_pq} holds, then it is easy to check that
\[
(\lambda+2\mu)\left(\Delta q_1+\kappa_1^2
q_1\right)=-\zeta_1,\quad
\mu\left(\Delta q_2+\kappa_2^2 q_2\right)=-\zeta_2.
\]
Noting $\boldsymbol{p}=\nabla q_1+{\bf curl}q_2$, we obtain 
\begin{eqnarray*}
&&
\mu\Delta\boldsymbol{p}+(\lambda+\mu)\nabla\nabla\cdot\boldsymbol{p}
+\omega^2\boldsymbol{p}\\
&&= \mu\nabla\left(\Delta q_1\right)+\mu{\bf
curl}\Delta q_2+(\lambda+\mu)\nabla\Delta q_1
+\omega^2\nabla q_1+\omega^2{\bf curl}q_2\\
&&=(\lambda+2\mu)\nabla\left(\Delta q_1+\kappa_1^2 q_1\right)
+\mu{\bf curl}\left(\Delta q_2+\kappa_2^2
q_2\right)\\
&&=-\nabla\zeta_1-{\bf curl}\zeta_2=-\boldsymbol{\xi}.		
\end{eqnarray*}

Next is to verify that the boundary condition on $y=b$. Assume that $\boldsymbol
p$ admits the Fourier expansion
$\boldsymbol{p}=\sum\limits_{n\in\mathbb{Z}}(p_{1}^{(n)}(y),
p_{2}^{(n)}(y))^\top e^{{\rm i}\alpha_n x}$. It follows from the
Helmholtz decomposition that 
\[
\begin{bmatrix}
p_1^{(n)}(y) \\[5pt] 
p_2^{(n)}(y)
\end{bmatrix}=
\begin{bmatrix}
{\rm i}\alpha_n q_{1}^{(n)}(y)+q^{(n)'}_2(y)\\[5pt]
q^{(n)'}_1(y)-{\rm i}\alpha_n q^{(n)}_2(y)
\end{bmatrix},
\]
which gives 
\[
 \begin{bmatrix}
p_1^{(n)'}(y) \\[5pt] 
p_2^{(n)'}(y)
\end{bmatrix}=
\begin{bmatrix}
{\rm i}\alpha_n q^{(n)'}_1(y)+ q^{(n)''}_2(y)\\[5pt] 
q^{(n)''}_1(y)-{\rm i}\alpha_n q^{(n)'}_2(y)
\end{bmatrix}.
\]

A straightforward calculation yields that 
\begin{eqnarray*}
\mathscr B\boldsymbol{p} &=&
\mu\partial_y\boldsymbol{p}+(\lambda+\mu)(0,1)^\top \nabla\cdot\boldsymbol{p}\\
&=&\sum\limits_{n\in\mathbb{Z}} \begin{bmatrix}
\mu\left({\rm i}\alpha_n
q^{(n)'}_1(y)+q^{(n)''}_2(y)\right)\\
(\lambda+\mu){\rm i}\alpha_n\left({\rm
i}\alpha_n q^{(n)}_1(y)+q^{(n)'}_2(y)\right)
+(\lambda+2\mu)\left(q^{(n)''}_1(y)-{\rm
i}\alpha_n q^{(n)'}_2(y)\right)
\end{bmatrix}e^{{\rm i}\alpha_n x}\\
&=&\sum\limits_{n\in\mathbb{Z}} \begin{bmatrix}
\mu\left({\rm i}\alpha_n
q^{(n)'}_1(y)+q^{(n)''}_2(y)\right)\\
(\lambda+2\mu)
q^{(n)''}_1(y)-(\lambda+\mu)\alpha_n^2 q^{(n)}_1(y)-{\rm i}\mu\alpha_n
q^{(n)'}_2(y)
\end{bmatrix}e^{{\rm i}\alpha_n x}.
\end{eqnarray*}
Evaluating the above equations at $y=b$, we get 
\begin{equation*}
\mathscr B\boldsymbol{p}|_{y=b}=\sum\limits_{n\in\mathbb{Z}}
\begin{bmatrix}
{\rm i}\mu\alpha_n q^{(n)'}_1(b)+\mu q^{(n)''}_2(b)\\[5pt]
(\lambda+2\mu)q^{(n)''}_1(b)-(\lambda+\mu)\alpha_n^2 q^{(n)}_1(b)-{\rm
i}\mu\alpha_n q^{(n)'}_2(b)
\end{bmatrix}e^{{\rm i}\alpha_n x}.
\end{equation*}
Noting $\zeta_j^{(n)}(b)=0$, we have from \eqref{Dual_pq} that 
$q^{(n)''}_j(b)=-(\kappa_j^2-\alpha_n^2)q^{(n)}_j(b)$. Hence
\begin{eqnarray*}
\mathscr B\boldsymbol{p}|_{y=b}=\sum\limits_{n\in\mathbb{Z}}\begin{bmatrix}
\mu\alpha_n\overline{\beta_1^{(n)}} &
-\omega^2+\mu\alpha_n^2\\[5pt]
\mu\alpha_n^2-\omega^2 &
-\mu\alpha_n\overline{\beta_2^{(n)}}
\end{bmatrix}\begin{bmatrix}
q^{(n)}_1(b) \\[5pt]  
q^{(n)}_2(b)
\end{bmatrix}e^{{\rm i}\alpha_n x}.
\end{eqnarray*}

On the other hand, we have 
\begin{eqnarray*}
\mathscr{T}^{*}\boldsymbol{p} &=&
\sum\limits_{n\in\mathbb{Z}}(M^{(n)})^* \boldsymbol{p}^{(n)}(b)e^{{\rm
i}\alpha_n x}\\
&=& \sum\limits_{n\in\mathbb{Z}} -\frac{\rm
i}{\overline{\chi}_n}\begin{bmatrix}
\omega^2\overline{\beta_1^{(n)}} &
\omega^2\alpha_n-\mu\alpha_n\overline{\chi}_n\\[5pt]
\mu\alpha_n\overline{\chi}_n-\omega^2\alpha_n & \omega^2\overline{\beta_2^{(n)}}
\end{bmatrix}\boldsymbol{p}^{(n)}(b)e^{{\rm
i}\alpha_n x}\\
&=&\sum\limits_{n\in\mathbb{Z}} -\frac{\rm
i}{\overline{\chi}_n}\begin{bmatrix}
\omega^2\overline{\beta_{1}^{(n)}} &
\omega^2\alpha_n-\mu\alpha_n\overline{\chi}_n\\[5pt]
\mu\alpha_n\overline{\chi}_n-\omega^2\alpha_n & \omega^2\overline{\beta_2^{(n)}}
\end{bmatrix}\begin{bmatrix}
{\rm i}\alpha_n & -{\rm
i}\overline{\beta_2^{(n)}}\\[5pt]
-{\rm i}\overline{\beta_1^{(n)}} &
-{\rm i}\alpha_n
\end{bmatrix}\begin{bmatrix}
q^{(n)}_1(b) \\[5pt]  
q^{(n)}_2(b)
\end{bmatrix}e^{{\rm i}\alpha_n x}\\
&=&\sum\limits_{n\in\mathbb{Z}}\begin{bmatrix}
\mu\alpha_n\overline{\beta_1^{(n)}} &
-\omega^2+\mu\alpha_n^2\\[5pt]
\mu\alpha_n^2-\omega^2 &
-\mu\alpha_n\overline{\beta_2^{(n)}}
\end{bmatrix}\begin{bmatrix}
q^{(n)}_1(b) \\[5pt]  
q^{(n)}_2(b)
\end{bmatrix}e^{{\rm i}\alpha_n x},
\end{eqnarray*}
which shows $\mathscr B\boldsymbol{p}=\mathscr{T}^{*}\boldsymbol{p}$ and 
completes the proof.
\end{proof}

It follows from the classic theory of second order differential equations
that the solution of the system
\[
\begin{cases}
q^{(n)''}_j(y)-|\beta_j^{(n)}|^2
q^{(n)}_j(y)=-c_j\zeta_j^{(n)}(y),\\
q^{(n)}_j(b')=q^{(n)}_j(b'),\\
q^{(n)'}_j(b)=-|\beta_j^{(n)}| q^{(n)}_j(b)
\end{cases}
\]
is 
\begin{eqnarray}\label{qjn}
q^{(n)}_j(y) = \frac{1}{2|\beta_j^{(n)}|}\Bigg\{
-c_j\int_{b}^{y}e^{|\beta_j^{(n)}|(y-s)}\zeta_j^{(n)}(s){\rm d}s
+c_j\int_{b'}^{y}e^{|\beta_j^{(n)}|(s-y)}\zeta_j^{(n)}(s){\rm d}s\notag\\
-c_j\int_{b'}^{b}e^{|\beta_j^{(n)}|(2b'-y-s)}\zeta_j^{(n)}(s){\rm d}s
+2|\beta_j^{(n)}|e^{|\beta_j^{(n)}|(b'-y)} q^{(n)}_j(b')
\Bigg\}.
\end{eqnarray}

\begin{lemma}\label{Lemma8}
Let $\boldsymbol{p}=(p_1, p_2)^\top$ be the solution of the dual problem problem
\eqref{DualProblem}. For sufficiently large $|n|$, the following estimate
hold
\begin{eqnarray*}
\left|p_j^{(n)}(b)\right|\lesssim |n|e^{|\beta_2^{(n)}|(b'-b)}\left(|p_1^{(n)}(b')|+|p_2^{(n)}(b')|\right)
+\frac{1}{|n|}\left(\|\xi_1^{(n)}\|_{L^{\infty}(b',
b)}+\|\xi_2^{(n)}\|_{L^{\infty}(b', b)}\right),
\end{eqnarray*}	
where $p_j^{(n)}$ is the Fourier coefficient of $p_j, j=1, 2$.
\end{lemma}

\begin{proof}
Evaluating \eqref{qjn} at $y=b$ yields 
\begin{eqnarray}\label{Lemma8-s1}
q^{(n)}_j(b)=\frac{1}{2|\beta_j^{(n)}|}\Bigg\{
c_j\int_{b'}^{b}e^{|\beta_j^{(n)}|(s-b)}\zeta_j^{(n)}(s){\rm d}s
-c_j\int_{b'}^{b}e^{|\beta_j^{(n)}|(2b'-b-s)}\zeta_j^{(n)}(s){\rm d}s\notag\\
+2|\beta_j^{(n)}|e^{|\beta_j^{(n)}|(b'-b)} q^{(n)}_j(b')
\Bigg\}.
\end{eqnarray}
Taking the derivative of $q^{(n)}_j$ with respect to $y$ in \eqref{qjn} and then
evaluating at $y=b'$, we have 
\begin{eqnarray*}
q^{(n)'}_j(b')=c_j\int_{b'}^{b} 
e^{|\beta_j^{(n)}|(b'-s)}\zeta_j^{(n)}(s){\rm d}s
-|\beta_j^{(n)}| q_{1}^{(n)}(b'),\quad j=1, 2,
\end{eqnarray*}
which is equivalent to
\begin{equation*}
\begin{bmatrix}
q^{(n)'}_1(b') \\[5pt]
q^{(n)'}_2(b')
\end{bmatrix}=
\begin{bmatrix}
-|\beta_1^{(n)}| & 0\\[5pt]
0 & -|\beta_2^{(n)}|
\end{bmatrix}
\begin{bmatrix}
q^{(n)}_1(b') \\[5pt]  
q^{(n)}_2(b')
\end{bmatrix}+
\begin{bmatrix}
\hat{\zeta}^{(n)}_1 \\[5pt]  
\hat{\zeta}^{(n)}_2
\end{bmatrix},
\end{equation*}
where
\[
 \hat{\zeta}^{(n)}_j=c_j\int_{b'}^{b}e^{|\beta_j^{(n)}|(b'-s)}\zeta_j^{(n)}(s){
\rm d}s. 
\]

It follows from Lemma \ref{Lemma7} and the Helmholtz
decomposition $\boldsymbol{p}=\nabla q_1+ {\bf curl}q_2$ that 
\begin{equation*}
\begin{bmatrix}
p_1^{(n)}(b') \\[5pt] 
p_2^{(n)}(b')
\end{bmatrix}=\begin{bmatrix}
{\rm i}\alpha_n q^{(n)}_1(b')+q^{(n)'}_2(b') \\[5pt]
q^{(n)'}_1(b')-{\rm i}\alpha_n  q^{(n)}_2(b')
\end{bmatrix}=\begin{bmatrix}
{\rm i}\alpha_n & -|\beta_2^{(n)}|\\[5pt]
-|\beta_1^{(n)}| & -{\rm i}\alpha_n
\end{bmatrix}\begin{bmatrix}
q^{(n)}_1(b') \\[5pt]  
q^{(n)}_2(b')
\end{bmatrix}+
\begin{bmatrix}
\hat{\zeta}_2^{(n)} \\[5pt]
\hat{\zeta}_1^{(n)}
\end{bmatrix},
\end{equation*}
which gives 
\begin{equation*}
\begin{bmatrix}
q^{(n)}_1(b') \\[5pt]
q^{(n)}_2(b')
\end{bmatrix}=\frac{1}{\chi_n}
\begin{bmatrix}
-{\rm i}\alpha_n & |\beta_2^{(n)}|\\[5pt]
|\beta_1^{(n)}| & {\rm i}\alpha_n
\end{bmatrix}\begin{bmatrix}
p_1^{(n)}(b') \\[5pt]
p_2^{(n)}(b')
\end{bmatrix}-\frac{1}{\chi_n}
\begin{bmatrix}
-{\rm i}\alpha_n & |\beta_2^{(n)}|\\
|\beta_1^{(n)}| & {\rm i}\alpha_n
\end{bmatrix}
\begin{bmatrix}
\hat{\zeta}_2^{(n)} \\[5pt] 
\hat{\zeta}_1^{(n)}
\end{bmatrix}.
\end{equation*}

Substituting the boundary condition 
\begin{equation*}
\begin{bmatrix}
q^{(n)'}_1 (b) \\[5pt]  
q^{(n)'}_2 (b)
\end{bmatrix}=
\begin{bmatrix}
-|\beta_1^{(n)}| & 0\\[5pt]
0 & -|\beta_2^{(n)}|
\end{bmatrix}\begin{bmatrix}
q^{(n)}_1(b) \\[5pt]  
q^{(n)}_2(b)
\end{bmatrix}
\end{equation*}
into the Helmholtz decomposition $\boldsymbol{p}=\nabla q_1+ {\bf curl}q_2$,
i.e., 
\[
\begin{bmatrix}
p_1^{(n)}(b) \\[5pt] 
p_2^{(n)}(b)
\end{bmatrix}=\begin{bmatrix}
{\rm i}\alpha_n q^{(n)}_1(b)+q^{(n)'}_2(b) \\[5pt]
q^{(n)'}_1(b)-{\rm i}\alpha_n  q^{(n)}_2(b)
\end{bmatrix},
\]
we obtain 
\begin{eqnarray*}
\begin{bmatrix}
p_1^{(n)}(b) \\[5pt] 
p_2^{(n)}(b)
\end{bmatrix}=\begin{bmatrix}
{\rm i}\alpha_n & -|\beta_2^{(n)}|\\[5pt]
-|\beta_1^{(n)}| & -{\rm i}\alpha_n
\end{bmatrix}\begin{bmatrix}
q^{(n)}_1(b) \\[5pt]  
q^{(n)}_2(b)
\end{bmatrix}.
\end{eqnarray*}
By \eqref{Lemma8-s1},
\begin{equation*}
\begin{bmatrix}
q^{(n)}_1(b) \\[5pt]
q^{(n)}_2(b)
\end{bmatrix}=\begin{bmatrix}
e^{|\beta_1^{(n)}|(b'-b)} & 0\\[5pt]
0 & e^{|\beta_2^{(n)}|(b'-b)}
\end{bmatrix}\begin{bmatrix}
q^{(n)}_1(b') \\[5pt]  
q^{(n)}_2(b')
\end{bmatrix}+\begin{bmatrix}
\eta_1^{(n)} \\[5pt] 
\eta_2^{(n)}
\end{bmatrix},
\end{equation*}
where
\[
 \eta_j^{(n)}=\frac{c_j}{2|\beta_j^{(n)}|}
\int_{b'}^{b}\left(e^{|\beta_j^{(n)}|(s-b)}-e^{|\beta_j^{(n)}|(2b'-b-s)}
\right)\zeta_j^{(n)}(s){\rm d}s. 
\]
Combining the above equations leads to 
\begin{eqnarray*}
\begin{bmatrix}
p_1^{(n)}(b) \\[5pt] 
p_2^{(n)}(b)
\end{bmatrix}&=&
\begin{bmatrix}
{\rm i}\alpha_n & -|\beta_2^{(n)}|\\[5pt]
-|\beta_1^{(n)}| & -{\rm i}\alpha_n
\end{bmatrix}\begin{bmatrix}
e^{|\beta_1^{(n)}|(b'-b)} & 0\\[5pt]
0 & e^{|\beta_2^{(n)}|(b'-b)}
\end{bmatrix}\begin{bmatrix}
q^{(n)}_1(b') \\[5pt]  
q^{(n)}_2(b')
\end{bmatrix}\\
&&+\begin{bmatrix}
{\rm i}\alpha_n & -|\beta_2^{(n)}|\\[5pt]
-|\beta_1^{(n)}| & -{\rm i}\alpha_n
\end{bmatrix}\begin{bmatrix}
\eta_1^{(n)} \\[5pt] 
\eta_2^{(n)}
\end{bmatrix}\\
&=&P^{(n)}\begin{bmatrix}
p_1^{(n)}(b') \\[5pt] 
p_2^{(n)}(b')
\end{bmatrix}
-P^{(n)}\begin{bmatrix}
\hat{\zeta}_{2}^{(n)} \\[5pt] 
\hat{\zeta}_{1}^{(n)}
\end{bmatrix}+\begin{bmatrix}
{\rm i}\alpha_n & -|\beta_2^{(n)}|\\[5pt]
-|\beta_1^{(n)}| & -{\rm i}\alpha_n
\end{bmatrix}\begin{bmatrix}
\eta_1^{(n)} \\[5pt] 
\eta_2^{(n)}
\end{bmatrix},
\end{eqnarray*}
where $P^{(n)}$ is defined in \eqref{PN}.

Recall that 
\begin{eqnarray*}
|\zeta_j^{(n)}(s)|\lesssim
\frac{1}{|\alpha_n|}\left(\|\xi_1^{(n)}\|_{L^{\infty}(b',
b)}+\|\xi_2^{(n)}\|_{L^{\infty}(b', b)}\right)
e^{|\alpha_n|(b-s)}.
\end{eqnarray*}
Since $s-b\geq 2b'-b-s$ and $|\alpha_n|\sim |n|, |\beta_j^{(n)}|\sim |n|$ for sufficiently large
$|n|$, we have from \eqref{AsymMn} and the mean-value theorem that 
\begin{eqnarray*}
|\eta^{(n)}_j| &\lesssim&
\left(\|\xi_1^{(n)}\|_{L^{\infty}(b', b)}+\|\xi_2^{(n)}\|_{L^{\infty}(b',
b)}\right)
\frac{1}{|\beta_j^{(n)}|}\left|\int_{b'}^{b}e^{|\beta_j^{(n)}|(s-b)}\frac{1}{
|\alpha_n|}e^{|\alpha_n|(b-s)}{\rm d}s\right|\\
&=& \left(\|\xi_1^{(n)}\|_{L^{\infty}(b',
b)}+\|\xi_2^{(n)}\|_{L^{\infty}(b', b)}\right)
\frac{1}{|\alpha_n| |\beta_j^{(n)}|}\frac{-1}{|\alpha_n|-|\beta_j^{(n)}|}
\left(1-e^{\left(|\alpha_n|-|\beta_j^{(n)}|\right)(b-b')}
\right)\\
&\lesssim& \frac{1}{n^2}\left(\|\xi_1^{(n)}\|_{L^{\infty}(b',
b)}+\|\xi_2^{(n)}\|_{L^{\infty}(b', b)}\right).
\end{eqnarray*}
Combining the above estimates yields 
\begin{eqnarray*}
\left|{\rm i}\,\alpha_n\eta_1^{(n)}-|\beta_2^{(n)}|\eta_2^{(n)}\right|, \quad 
\left|-|\beta_1^{(n)}|\eta_1^{(n)}-{\rm i}\,\alpha_n\eta_2^{(n)}\right|\lesssim \frac{1}{|n|}
\left(\|\xi_1^{(n)}\|_{L^{\infty}(b',
b)}+\|\xi_2^{(n)}\|_{L^{\infty}(b', b)}\right).
\end{eqnarray*}

Following the similar steps of the estimate for $\eta_j^{(n)}$, we can show that 
\begin{eqnarray*}
|\hat{\zeta}_j^{(n)}| &\lesssim& 
\left(\|\xi_1^{(n)}\|_{L^{\infty}(b',
b)}+\|\xi_2^{(n)}\|_{L^{\infty}(b', b)}\right)
\int_{b'}^{b}e^{|\beta_j^{(n)}|(b'-s)}e^{|\alpha_n|
(b-s)}\frac{1}{|\alpha_n|}{\rm d}s\\
&\lesssim& \frac{1}{|\alpha_n|(|\alpha_n|+|\beta_j^{(n)}|)}
\left(\|\xi_1^{(n)}\|_{L^{\infty}(b',
b)}+\|\xi_2^{(n)}\|_{L^{\infty}(b', b)}\right)
\left|e^{|\beta_j^{(n)}|(b'-b)}-e^{|\alpha_n| (b-b')}\right|\\
&\lesssim& \frac{1}{n^2}\left(\|\xi_1^{(n)}\|_{L^{\infty}(b',
b)}+\|\xi_2^{(n)}\|_{L^{\infty}(b', b)}\right)
e^{|\alpha_n| (b-b')},
\end{eqnarray*}
which gives
\begin{eqnarray*}
\left|P^{(n)} \begin{bmatrix}
\hat{\zeta}_1^{(n)} \\[5pt]
\hat{\zeta}_2^{(n)}
\end{bmatrix}\right| 
&\lesssim&  |n|e^{-|\beta_2^{(n)}|(b-b')}\frac{1}{n^2}
\left(\|\xi_1^{(n)}\|_{L^{\infty}(b',
b)}+\|\xi_2^{(n)}\|_{L^{\infty}(b', b)}\right)
e^{|\alpha_n| (b-b')}\\
&\lesssim& \frac{1}{|n|}e^{\left(|\alpha_n|-|\beta_2^{(n)}|\right)
(b-b')} \left(\|\xi_1^{(n)}\|_{L^{\infty}(b',
b)}+\|\xi_2^{(n)}\|_{L^{\infty}(b', b)}\right).
\end{eqnarray*}		
Since for sufficiently large $|n|$, we have 
\begin{eqnarray*}
|\alpha_n|-|\beta_2^{(n)}| =
|\alpha_n|-\left(\alpha_n^2-\kappa_2^2\right)^{1/2}
=\frac{\kappa_2^2}{|\alpha_n|+\left(\alpha_n^2-\kappa_2^2\right)^{1/2}}
\sim \frac{1}{|n|}.
\end{eqnarray*}	
Hence
\begin{eqnarray*}
\left|P^{(n)} \begin{bmatrix}
\hat{\zeta}_1^{(n)} \\[5pt] 
\hat{\zeta}_2^{(n)}
\end{bmatrix}\right| 
\lesssim \frac{1}{|n|}\left(\|\xi_1^{(n)}\|_{L^{\infty}(b',
b)}+\|\xi_2^{(n)}\|_{L^{\infty}(b', b)}\right),
\end{eqnarray*}
which proves
\begin{eqnarray*}
\left|p_j^{(n)}(b)\right|\lesssim
|n|e^{|\beta_2^{(n)}|(b'-b)}\left(|p_1^{(n)}(b')|+|p_2^{(n)}(b')|\right)
+\frac{1}{|n|}\left(\|\xi_1^{(n)}\|_{L^{\infty}(b',
b)}+\|\xi_2^{(n)}\|_{L^{\infty}(b', b)}\right).
\end{eqnarray*}
The proof is completed. 
\end{proof}

Taking $\boldsymbol{v}=\boldsymbol{\xi}$ in \eqref{DualProblem}, we have
\begin{equation}\label{Dualxixi}
\| \boldsymbol{\xi}\|_{\boldsymbol{L}^2(\Omega)}^2=a(\boldsymbol{\xi}, \boldsymbol{p})
-\int_{\Gamma}\left(\mathscr{T}-\mathscr{T}_N\right)\boldsymbol{\xi}\cdot\overline{\boldsymbol{p}}\,{\rm d}s
+\int_{\Gamma}\left(\mathscr{T}-\mathscr{T}_N\right)\boldsymbol{\xi}\cdot\overline{\boldsymbol{p}}\,{\rm d}s.
\end{equation}
By Lemma \ref{Lemma8}, we obtain 
\begin{eqnarray}
&&
\left|\int_{\Gamma}\left(\mathscr{T}-\mathscr{T}_N\right)\boldsymbol{\xi}
\cdot\overline{ \boldsymbol{p}}\,{\rm d}s \right|
\leq \Lambda\sum\limits_{|n|>N}
\left|\left(M^{(n)}\boldsymbol{\xi}_n(b)\right)\cdot
\overline{\boldsymbol{p}}_n(b)\right| \notag\\
&&\lesssim \Lambda \sum\limits_{|n|>N} |n|
\left(|\xi_1^{(n)}(b)|+|\xi_2^{(n)}(b)|\right)
\left(|p_1^{(n)}(b)|+|p_2^{(n)}(b)|\right)	\notag \\
&& \lesssim N^{-1}
\left[\sum\limits_{|n|>N}(1+n^2)^{1/2}\left(|\xi_1^{(n)}(b)|+|\xi_2^{(n)}
(b)|\right)^2\right]^{1/2}	
\left[\sum\limits_{|n|>N} |n|^3
\left(|p_1^{(n)}(b)|+|p_2^{(n)}(b)|\right)^2\right]^{1/2}\notag \\
&& \lesssim N^{-1} \|\boldsymbol{\xi}\|_{\boldsymbol H^{1/2}(\Gamma)}	
\left[\sum\limits_{|n|>N} |n|^3
\left(|p_1^{(n)}(b)|^2+|p_2^{(n)}(b)|^2\right)\right]^{1/2}\notag\\
&& \lesssim N^{-1} \|\boldsymbol{\xi}\|_{\boldsymbol H^{1}(\Omega)}	
\left[\sum\limits_{|n|>N} |n|^3
\left(|p_1^{(n)}(b)|^2+|p_2^{(n)}(b)|^2\right)\right]^{1/2}\label{P_2}.
\end{eqnarray}	

Following the similar proof in \cite[eq. (30)]{JLLZ-JSC-2017}, we may show that 
\begin{equation}\label{xixi}
\|\xi_j^{(n)}\|^2_{L^{\infty}(b', b)} \leq
\left(\frac{2}{\delta}+|n|\right)\|\xi_j^{(n)}(t)\|^2_{L^2(b', b)}
+|n|^{-1}\|\xi_j^{(n)}{}'(t)\|^2_{L^2(b', b)}.
\end{equation}
It follows from the Cauchy--Schwarz inequality that 
\begin{eqnarray*}
&& \sum\limits_{|n|>N} |n|^3
\left(|p_1^{(n)}(b)|^2+|p_2^{(n)}(b)|^2\right)\\
&& \lesssim \sum\limits_{|n|>N} |n|^3
\left\{n^2e^{2|\beta_2^{(n)}|(b'-b)}\left(|p_1^{(n)}(b')|^2+|p_2^{(n)}
(b')|^2\right)+\frac{1}{|n|^2}\left(\|\xi_1^{(n)}\|^2_{L^{\infty}(b',
b)}+\|\xi_2^{(n)}\|^2_{L^{\infty}(b', b)}\right)\right\}\\
&&\lesssim\sum\limits_{|n|>N} |n|^5
e^{2|\beta_2^{(n)}|(b'-b)}\left(|p_1^{(n)}(b')|^2+|p_2^{(n)}(b')|^2\right)
+\sum\limits_{|n|>N} |n|\left(\|\xi_1^{(n)}\|^2_{L^{\infty}(b',
b)}+\|\xi_2^{(n)}\|^2_{L^{\infty}(b', b)}\right)\\
&& := I_1+I_2.	
\end{eqnarray*}
Noting that the function $t^{4} e^{-2t}$ is bounded on $(0, +\infty)$, we have
\[
I_1\lesssim
\max\limits_{|n|>N}\left(n^{4}e^{2|\beta_2^{(n)}|(b'-b)}\right)
\sum\limits_{|n|>N}
|n|\left(|p_1^{(n)}(b')|^2+|p_2^{(n)}(b')|^2\right)
\lesssim \|\boldsymbol{p}\|_{\boldsymbol H^{1/2}(\Gamma')}^2
\lesssim \|\boldsymbol{\xi}\|_{\boldsymbol H^{1}(\Omega)}^2.
\]
Substituting \eqref{xixi} into $I_2$, we get
\begin{eqnarray*}
I_2 &\lesssim& \sum\limits_{|n|>N}
\left[|n|\left(\frac{2}{\delta}+|n|\right)\left(\|\xi_1^{(n)}\|_{L^{2}(b', b)}^2
+\|\xi_2^{(n)}\|_{L^{2}(b', b)}^2 \right)+
\left(\|\xi_1^{(n)'}\|_{L^{2}(b', b)}^2
+\|\xi_2^{(n)'}\|_{L^{2}(b', b)}^2
\right)\right]\\
&\leq& \sum\limits_{|n|>N}
\left[\left(\frac{2}{\delta}|n|+n^2\right)\|\boldsymbol{\xi}_n\|_{\boldsymbol L^{2}(b',
b)}^2 + \|\boldsymbol{\xi}_n^\prime\|_{\boldsymbol L^{2}(b', b)}^2\right]. 
\end{eqnarray*}

A simple calculation yields
\begin{eqnarray*}
\|\xi_j^{(n)}\|^2_{H^{1}(\Omega')} &=&
\Lambda\sum\limits_{n\in\mathbb{Z}} \int_{b'}^{b}
\left[\left(1+\alpha_n^2\right)|\xi_j^{(n)}(y)|^2+|\xi_j^{(n)'}(y)|^2\right]
{\rm d}y.
\end{eqnarray*}
It is easy to note that
\[
\frac{2}{\delta}|n|+n^2\lesssim 1+\alpha_n^2.
\]
Then
\[
I_2\lesssim  \|\boldsymbol{\xi}\|^2_{\boldsymbol H^{1}(\Omega')}\leq
\|\boldsymbol{\xi}\|^2_{\boldsymbol H^{1}(\Omega)}.
\]
Therefore,
\begin{equation}\label{xixixi}
\sum\limits_{|n|>N}|n|^3
\left(|p_1^{(n)}(b)|+|p_2^{(n)}(b)|\right)^2\lesssim
\|\boldsymbol{\xi}\|^2_{\boldsymbol H^{1}(\Omega)}.
\end{equation}
Plugging \eqref{xixixi} to \eqref{P_2}, we obtain
\begin{equation}\label{4thTerm}
|\int_{\Gamma}\left(\mathscr{T}-\mathscr{T}_N\right)\boldsymbol{\xi}
\cdot\overline{\boldsymbol{p}}\,{\rm d}s |\lesssim
\frac{1}{N}\|\boldsymbol{\xi}\|^2_{\boldsymbol{H}^{1}(\Omega)}.
\end{equation}

Now, we prove Theorem \ref{Main_Result}.

\begin{proof}
By Lemma \ref{Posterior: Lemma3}, Lemma \ref{Posterior: FirstTwo}, and Lemma
\ref{Posterior: Lemma5}, we have
\begin{eqnarray*}
\vvvert\boldsymbol{\xi}\vvvert^2_{\boldsymbol{H}^{1}(\Omega)} &=& \Re a(\boldsymbol{\xi}, \boldsymbol{\xi})
+\Re\int_{\Gamma} \left(\mathscr{T}-\mathscr{T}_N\right)\boldsymbol{\xi}\cdot\overline{\boldsymbol
{\xi}}{\rm d}s+2\omega^2 \int_{\Omega}
\boldsymbol{\xi}\cdot\overline{\boldsymbol{\xi}}{\rm d}\boldsymbol{x}
+\Re\int_{\Gamma}
\mathscr{T}_N\boldsymbol{\xi}\cdot\overline{\boldsymbol{\xi}}{\rm d}s\\	
&\leq& C_1\left[\left(\sum\limits_{T\in M_h} \eta^2_{T}\right)^{1/2}
+\max_{|n|>N}\left(|n|e^{|\beta_2^{(n)}|(b'-b)}\right) \|\boldsymbol{u}^{\rm
inc}\|_{\boldsymbol{H}^{1}(\Omega)}\right] 	
\|\boldsymbol{\xi}\|_{\boldsymbol{H}^{1}(\Omega)}\\
&&\qquad
+\left(C_2+C(\delta)\right)\|\boldsymbol{\xi}\|^2_{\boldsymbol{L}^2(\Omega)}
+\delta\|\boldsymbol{\xi}\|^2_{\boldsymbol{H}^{1}(\Omega)},
\end{eqnarray*}		
where $C_1, C_2, C(\delta)$ are positive constants. From \eqref{Equal_norm}, by choosing $\delta$ such that $\frac{\delta}{\min(\mu,
\omega^2)}<\frac{1}{2}$, we get
\begin{eqnarray}\label{Estimate_H1}
\vvvert\boldsymbol{\xi}\vvvert^2_{\boldsymbol{H}^{1}(\Omega)}
\leq  2C_1\left[\left(\sum\limits_{T\in M_h}
\eta^2_{T}\right)^{1/2}
+\max_{|n|>N}\left(|n|e^{|\beta_2^{(n)}|(b'-b)}\right)\|\boldsymbol{u}^{\rm
inc}\|_{\boldsymbol{H}^{1}(\Omega)}\right] 	
\|\boldsymbol{\xi}\|_{\boldsymbol{H}^{1}(\Omega)}\notag\\
+2\left(C_2+C(\delta)\right)\|\boldsymbol{\xi}\|^2_{\boldsymbol{L}
^2(\Omega)}.
\end{eqnarray}
It follows from \eqref{Dualxixi} and \eqref{4thTerm} that 
\begin{eqnarray}
&&\|\boldsymbol{\xi}\|^{2}_{\boldsymbol{L}^2(\Omega)}
= b(\boldsymbol{\xi}, \boldsymbol{p})
+\int_{\Gamma}\left(\mathscr{T}-\mathscr{T}_N\right)\boldsymbol{\xi}
\cdot\overline{\boldsymbol{p}}\,{\rm d}s
-\int_{\Gamma}\left(\mathscr{T}-\mathscr{T}_N\right)\boldsymbol{\xi}
\cdot\overline{\boldsymbol{p}}\,{\rm d}s \notag\\
&&\lesssim	\left[\left(\sum\limits_{T\in M_h}
\eta^2_{T}\right)^{1/2}
+\max_{|n|>N}\left(|n|e^{|\beta_2^{(n)}|(b'-b)}\right)
\|\boldsymbol{u}^{\rm inc}\|_{\boldsymbol{H}^{1}(\Omega)}\right] 	
\|\boldsymbol{\xi}\|_{\boldsymbol{H}^{1}(\Omega)}+N^{-1}
\|\boldsymbol{\xi}\|^2_{\boldsymbol{H}^{1}(\Omega)}.	\label{Estimate_L2}
\end{eqnarray}
Taking sufficiently large $N$ such that 
$\frac{2\left(C_2+C(\delta)\right)}{N}\frac{1}{\min(\mu, \omega^2)}<1$ and 
substituting \eqref{Estimate_L2} into \eqref{Estimate_H1},  
we obtain
\begin{eqnarray*}
\vvvert\boldsymbol{u}-\boldsymbol{u}_N^h
\vvvert_{\boldsymbol{H}^{1}(\Omega)}\lesssim 
\left(\sum\limits_{T\in M_h} \eta^2_{T}\right)^{1/2}
+\max_{|n|>N}\left(|n|e^{|\beta_2^{(n)}|(b'-b)}\right)
\|\boldsymbol{u}^{\rm inc}\|_{\boldsymbol{H}^{1}(\Omega)}.
\end{eqnarray*}
The proof is completed by noting the equivalence of the norms $\vvvert
\cdot\vvvert_{\boldsymbol{H}^1(\Omega)}$ and  $\|\cdot\|_{\boldsymbol{H}^1(\Omega)}$. 
\end{proof}

\section{Numerical experiments}\label{section:ne}

In this section, we introduce the algorithmic implementation of the adaptive finite element DtN method and present two numerical examples to demonstrate the effectiveness of the proposed method. 

\subsection{Adaptive algorithm}

Our implementation is based on the FreeFem \cite{H-JNM-2012}. The first-order linear element is used to solve the problem. It is shown in Theorem \ref{Main_Result} that the a posteriori error consists of two parts: the finite element discretization error $\epsilon_h$ and the DtN operator truncation error $\epsilon_N$, where 
\begin{eqnarray}\label{epsilonN}
\epsilon_h = \left(\sum\limits_{K\in\mathcal M_h} \eta^2_{K}\right)^{1/2},\quad 
\epsilon_N =\max_{|n|>N}\left( |n|e^{-|\beta_2^{(n)}|(b-b')}\right)
\|\boldsymbol{u}^{\rm inc}\|_{\boldsymbol{H}^{1}(\Omega)}. 
\end{eqnarray}
In the implementation, we choose the parameters $b, b'$ and $N$ based on \eqref{epsilonN} to make sure that the DtN operator truncation error is smaller than the finite element discretization error. In the following numerical experiments, $b'$ is chosen such that $b'=\max_{x\in(0,\Lambda)}f(x)$ and $N$ is the smallest positive integer that makes $\epsilon_N\leq 10^{-8}$. The adaptive finite element algorithm is shown in Table 1.

\begin{table}\label{Table1}	
\caption{The adaptive finite element DtN method.}
\vspace{1ex}
\hrule \hrule 
\vspace{0.8ex} 
\begin{enumerate}		
\item Given the tolerance $\epsilon>0$ and the parameter $\tau\in(0,1)$.
\item Fix the computational domain $\Omega$ by choosing $b$.
\item Choose $b'$ and $N$ such that $\epsilon_N\leq 10^{-8}$.
\item Construct an initial triangulation $\mathcal M_h$ over $\Omega$ and compute
error estimators.
\item While $\epsilon_h>\epsilon$ do
\item \qquad refine mesh $\mathcal M_h$ according to the strategy that \\
\[\text{if } \eta_{\hat{K}}>\tau \max\limits_{K\in\mathcal M_h}
\eta_{K}, \text{ refine the element } \hat{K}\in\mathcal M_h, \]
\item \qquad denote refined mesh still by $\mathcal M_h$, solve the discrete
problem \eqref{TBC: variational3} on the new mesh $\mathcal M_h$,
\item \qquad compute the corresponding error estimators.
\item End while.
\end{enumerate}
\vspace{0.8ex}
\hrule\hrule
\vspace{0.8ex} 
\end{table}

\subsection{Numerical experiments}

We report two examples to illustrate the numerical performance of the proposed method. The first example concerns the scattering by a flat surface and has an exact solution; the second example is constructed such that the solution has corner singularity. 

{\em Example 1.} We consider the simplest periodic structure, a straight line, where the exact solution is available. Let $S=\{y=0\}$ and take the artificial boundary $\Gamma=\{y=0.25\}$. The space above the flat surface is filled with a homogenenous and isotropic elastic medium, which is characterized by the Lam\'{e} constants $\lambda=2$, $\mu=1$. The rigid surface is impinged by the compressional plane wave $\boldsymbol{u}^{\rm inc}=\boldsymbol{d}e^{{\rm i}\kappa_1 \boldsymbol{x}\cdot\boldsymbol{d}}$, where the incident angle is $\theta=\pi/3$. The compressional and shear wavenumbers are $\kappa_1=\omega/2$ and $\kappa_2=\omega$, respectively, where $\omega$ is the angular frequency. It can be verified that the exact solution is 
\begin{eqnarray*}
u(\boldsymbol{x})=\frac{1}{\kappa_1}
\begin{bmatrix} 
\alpha \\
 -\beta 
 \end{bmatrix}
e^{{\rm i}\left(\alpha x -\beta y\right)}
-\frac{1}{\kappa_1}\left(\frac{\alpha^2-\beta\gamma}{\alpha^2+\beta\gamma}\right)
\begin{bmatrix}\alpha \\ \beta\end{bmatrix}
e^{{\rm i}\left(\alpha x + \beta y\right)}
-\frac{1}{\kappa_1}\left(\frac{2\alpha \beta}{\alpha^2+\beta\gamma}\right)
\begin{bmatrix}\gamma \\ -\alpha\end{bmatrix}
e^{{\rm i}\left(\alpha x +\gamma y\right)},
\end{eqnarray*}
where $\alpha=\kappa_1 \sin\theta, \beta=\kappa_1\cos\theta, \gamma=(\kappa_2^2-\alpha^2)^{1/2}$. The period $\Lambda=0.5$. Figure \ref{fig2} shows the curves of $\log e_h$ versus $\log {\rm DoF}_h$ with different angular frequencies, where $e_h=\|\boldsymbol u-\boldsymbol u_N^h\|_{\boldsymbol H^1(\Omega)}$ is the a priori error and ${\rm DoF}_h$ stands for the degree of freedom or the number of nodal points. It indicates that the meshes and the associated numerical complexity are  quasi-optimal, i.e., $e_h=\mathcal{O}({\rm DoF}_h^{-1/2})$ holds asymptotically.

\begin{figure}
\centering
\includegraphics[width=0.5\textwidth]{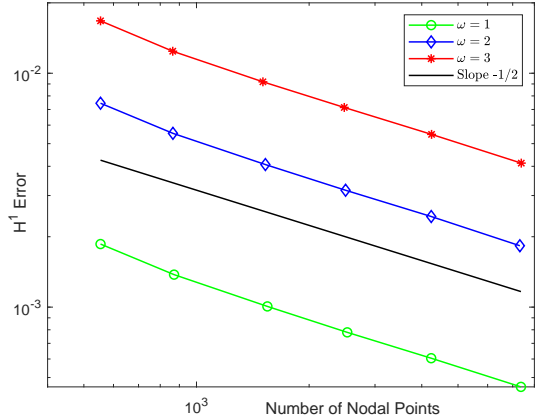}
\caption{Quasi-optimality of the a priori error estimates for Example 1.}
\label{fig2}
\end{figure}	

{\em Example 2.} This example concerns the scattering of the compressional plane wave by a piecewise linear surface, which has multiple sharp angles. The incident wave $\boldsymbol u^{\rm inc}$ and the parameters are chosen the same as Example 1, i.e., $b=0.25, \Lambda=0.5, \theta=\pi/3, \lambda=1, \mu=2$. Clearly, the solution has singularity around the corners of the surface. Since there is no exact solution for this example, we plot in Figure \ref{fig3} the curves of $\log\epsilon_h$ versus $\log{\rm DoF}_h$ at different angular frequencies, where $\epsilon_h$ is the a posteriori error. Again, it indicates that the meshes and the associated numerical complexity are quasi-optimal, i.e., $\epsilon_h=\mathcal{O}({\rm DoF}_h^{-1/2})$. Figure \ref{fig4} plots the contour of the magnitude of the numerical solution and its corresponding mesh at the angular frequency $\omega=2$. It is clear to note that the algorithm does capture the solution feature and adaptively refines the mesh around the corners where 
solution displays singularity.

\begin{figure}
\centering
\includegraphics[width=0.5\textwidth]{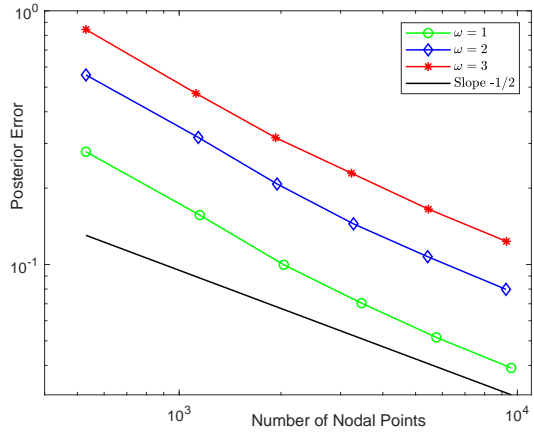}
\caption{Quasi-optimality of the a posteriori error estimates for Example 2.}
\label{fig3}
\end{figure}

\begin{figure}
\centering
\includegraphics[width=0.46\textwidth]{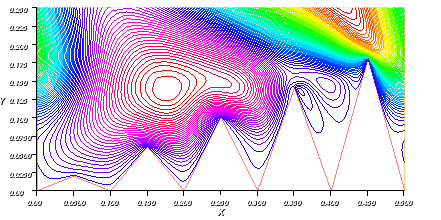}
\includegraphics[width=0.46\textwidth]{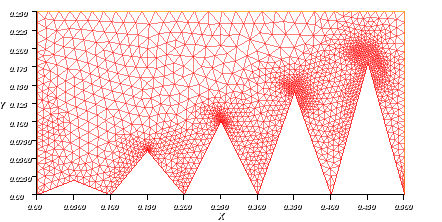}
\caption{The numerical solution of Example 2. (left) The contour plot of the magnitude of the solution; (right) The corresponding adaptively refined mesh.}
\label{fig4}
\end{figure}

\section{Conclusion}\label{section:c}

In this paper, we have presented an adaptive finite element DtN method for the elastic scattering problem in periodic structures. Based on the Helmholtz decomposition, a new duality argument is developed to obtain the a posteriori error estimate. It contains both the finite element discretization error and the DtN operator truncation error, which is shown to decay exponentially with respect to the truncation parameter. Numerical results show that the proposed method is effective and accurate. This work provides a viable alternative to the adaptive finite element PML method for solving the elastic surface scattering problem. It also enriches the range of choices available for solving wave propagation problems imposed in unbounded domains. One possible future work is to extend our analysis to the adaptive finite element DtN method for solving the 
three-dimensional elastic surface scattering problem, where a more complicated TBC needs to be considered.

\end{document}